\documentclass[12pt]{article}

\usepackage{etex}

\usepackage[top=1in, bottom=1in, left=1in, right=1in]{geometry}

\usepackage{amsmath}
\usepackage{amssymb}
\usepackage{amsthm}
\usepackage{array}
\usepackage{bbm}
\usepackage{bm}
\usepackage{cancel}
\usepackage{cmap}
\usepackage{csquotes}
\usepackage{enumerate}
\usepackage{enumitem}
\usepackage{esint}
\usepackage{fancyhdr}
\usepackage{listings}
\usepackage{longtable}
\usepackage{makeidx}
\usepackage{mathdots}
\usepackage{mathtools}
\usepackage{mathrsfs}
\usepackage{stackrel}
\usepackage{stmaryrd}
\usepackage{tabularx}
\usepackage{tikz}
\usepackage{ctable}
\usepackage{titlesec}
\usepackage{titletoc}
\usepackage{url}
\usepackage{verbatim}
\usepackage{wasysym}
\usepackage{wrapfig}
\usepackage{yhmath}
\usepackage[all,cmtip]{xy}

\usepackage{hyperref}
\usepackage{cleveref}

\usepackage{algorithm}
\usepackage{algpseudocode}
\algnewcommand\algorithmicinput{\textbf{Input: }}
\algnewcommand\INPUT{\State\algorithmicinput}
\algnewcommand\algorithmicinitialize{\textbf{Initialize: }}
\algnewcommand\INIT{\State\algorithmicinitialize}
\algnewcommand\algorithmicrun{\textbf{Run: }}
\algnewcommand\RUN{\State\algorithmicrun}
\algnewcommand\algorithmicupdate{\textbf{Update: }}
\algnewcommand\UPDATE{\State\algorithmicupdate}
\algnewcommand\algorithmicset{\textbf{Set: }}
\algnewcommand\SET{\State\algorithmicset}
\algnewcommand\algorithmicquery{\textbf{Query: }}
\algnewcommand\QUERY{\State\algorithmicquery}
\algnewcommand\algorithmicoutput{\textbf{Output: }}
\algnewcommand\OUTPUT{\State\algorithmicoutput}
\usepackage[backend=biber,style=alphabetic,maxbibnames=99]{biblatex}

\usetikzlibrary{calc,trees,positioning,arrows,chains,shapes.geometric,%
    decorations.pathreplacing,decorations.pathmorphing,shapes,%
    matrix,shapes.symbols,shadows,fadings}

\usepackage{nomencl} %
\makenomenclature
\RequirePackage{ifthen}

\newtheoremstyle{norm}
{12pt}
{12pt}
{}
{}
{\bf}
{:}
{.5em}
{}

\newtheorem{thm}{Theorem}[section]
\newtheorem*{thm*}{Theorem}

\newtheorem*{clm*}{Claim}
\newtheorem{conj}[thm]{Conjecture}
\newtheorem*{conj*}{Conjecture}
\newtheorem{cor}[thm]{Corollary}
\newtheorem{lem}[thm]{Lemma}
\newtheorem*{lem*}{Lemma}

\theoremstyle{norm}
\newtheorem{prb}[thm]{Problem}
\newtheorem*{prb*}{Problem}

\newtheorem{assm}[thm]{Assumption}

\newtheorem*{ax*}{Axiom}
\newtheorem{df}[thm]{Definition}
\newtheorem*{df*}{Definition}

\newtheorem*{ex*}{Example}

\newtheorem{expl}[thm]{Exploration}

\newtheorem*{pos*}{Postulate}

\newtheorem*{pr*}{Proposition}

\newtheorem*{qu*}{Question}
\newtheorem{rem}[thm]{Remark}
\newtheorem*{rem*}{Remark}

\usepackage[english]{babel} 
\usepackage{blindtext} 
\tikzstyle{prbox} = [draw=black, fill=blue!20, very thick,
    rectangle, inner sep=10pt, inner ysep=10pt]
\tikzstyle{thbox} = [draw=black,double, fill=blue!10, very thick,
    rectangle, inner sep=10pt, inner ysep=10pt]
\tikzstyle{cpbox} = [drop shadow={
    shadow scale=1}, draw=black, fill=green!10, very thick,
    rectangle, inner sep=10pt, inner ysep=10pt]
\tikzstyle{wrbox} = [drop shadow={
    shadow scale=1}, draw=black, fill=yellow!10, very thick,
    rectangle, inner sep=10pt, inner ysep=10pt]
\tikzstyle{hnbox} = [draw=black, fill=white, very thick,
    rectangle, inner sep=10pt, inner ysep=10pt]

\newcommand{\A}[0]{\mathbb{A}}

\newcommand{\E}[0]{\mathbb{E}}

\newcommand{\sF}[0]{\mathscr{F}}

\newcommand{\N}[0]{\mathbb{N}}

\newcommand{\Pj}[0]{\mathbb{P}}

\newcommand{\R}[0]{\mathbb{R}}

\newcommand{\T}[0]{\mathbb{T}}

\newcommand{\one}[0]{\mathbbm{1}}

\newcommand{\al}[0]{\alpha}
\newcommand{\be}[0]{\beta}
\newcommand{\ga}[0]{\gamma}

\newcommand{\de}[0]{\delta}
\newcommand{\De}[0]{\Delta}
\newcommand{\ep}[0]{\varepsilon}

\newcommand{\te}[0]{\theta}
\newcommand{\Te}[0]{\Theta}

\newcommand{\om}[0]{\omega}
\newcommand{\Om}[0]{\Omega}
\newcommand{\si}[0]{\sigma}

\newcommand{\nin}[0]{\not\in}

\newcommand{\subeq}[0]{\subseteq}

\newcommand{\bs}[0]{\backslash}
\newcommand{\iy}[0]{\infty}

\newcommand{\rc}[1]{\frac{1}{#1}}
\newcommand{\prc}[1]{\pa{\rc{#1}}}

\newcommand{\fc}[2]{\frac{#1}{#2}}
\newcommand{\sfc}[2]{\sqrt{\frac{#1}{#2}}}
\newcommand{\pf}[2]{\pa{\frac{#1}{#2}}}

\newcommand{\dd}[2]{\frac{d #1}{d #2}}

\newcommand{\ab}[1]{\left| {#1} \right|}

\newcommand{\ba}[1]{\left[ {#1} \right]}
\newcommand{\bc}[1]{\left\{ {#1} \right\}}

\newcommand{\ce}[1]{\left\lceil {#1}\right\rceil}
\newcommand{\fl}[1]{\left\lfloor {#1}\right\rfloor}

\newcommand{\pa}[1]{\left( {#1} \right)}

\newcommand{\set}[2]{\left\{{#1}:{#2}\right\}}

\newcommand{\ol}[1]{\overline{#1}}
\newcommand{\ul}[1]{\underline{#1}}
\newcommand{\ub}[2]{\underbrace{#1}_{#2}}

\newcommand{\wt}[1]{\widetilde{#1}}
\newcommand{\wh}[1]{\widehat{#1}}

\newcommand{\KL}[0]{\operatorname{KL}}

\newcommand{\TV}[0]{\operatorname{TV}}

\newcommand{\Var}[0]{\operatorname{Var}}

\providecommand{\cal}[1]{\mathcal{#1}}
\renewcommand{\cal}[1]{\mathcal{#1}}

\newcommand{\pull}[9]{
#1\ar@/_/[ddr]_{#2} \ar@{.>}[rd]^{#3} \ar@/^/[rrd]^{#4} & &\\
& #5\ar[r]^{#6}\ar[d]^{#8} &#7\ar[d]^{#9} \\}

\newcommand{\cmp}[9]{
\xymatrix{
#1 \ar[r]^{#4}{#5} \ar@/_2pc/[rr]^{#8}_{#9} & #2 \ar[r]^{#6}_{#7} & #3
}
}

\newcommand{\ha}[1]{\ar@{^(->}[#1]}
\newcommand{\ls}[1]{\ar@{-}[#1]}
\newcommand{\sj}[1]{\ar@{->>}[#1]}
\newcommand{\aq}[1]{\ar@{=}[#1]}
\newcommand{\acir}[1]{\ar@{}[#1]|-{\textstyle{\circlearrowright}}}
\newcommand{\acil}[1]{\ar@{}[#1]|-{\textstyle{\circlearrowleft}}}
\newcommand{\ard}[1]{\ar@{.>}[#1]}
\newcommand{\mt}[1]{\ar@{|->}[#1]}
\newcommand{\inm}[1]{\ar@{}[#1]|-{\in}}
\newcommand{\inr}{\ar@{}[d]|-{\rotatebox[origin=c]{-90}{$\in$}}}
\newcommand{\inl}{\ar@{}[u]|-{\rotatebox[origin=c]{90}{$\in$}}}

\newcommand{\minr}[2]{\min_{\scriptsize \begin{array}{c}{#1}\\{#2}\end{array}}}

\newcommand{\sumo}[2]{\sum_{#1=1}^{#2}}
\newcommand{\sumz}[2]{\sum_{#1=0}^{#2}}
\newcommand{\prodo}[2]{\prod_{#1=1}^{#2}}

\newcommand{\beq}[1]{\begin{equation}\llabel{#1}}
\newcommand{\eeq}[0]{\end{equation}}
\newcommand{\bal}[0]{\begin{align*}}
\newcommand{\eal}[0]{\end{align*}}
\newcommand{\ban}[0]{\begin{align}}
\newcommand{\ean}[0]{\end{align}}

\newcommand{\fixme}[1]{{\color{red}#1}}

\newcommand{\llabel}[1]{\label{#1}\text{\fixme{\tiny#1}}}

\newcommand{\arxiv}[1]{\url{http://www.arxiv.org/abs/#1}}

\newcommand{\vocab}[1]{\textbf{#1}} 

\allowdisplaybreaks[2]

\DeclareFontFamily{U}{wncy}{}
    \DeclareFontShape{U}{wncy}{m}{n}{<->wncyr10}{}
    \DeclareSymbolFont{mcy}{U}{wncy}{m}{n}
    \DeclareMathSymbol{\Sh}{\mathord}{mcy}{"58} 

\newcommand{\OPT}[0]{\mathrm{OPT}}

\newcommand{\amx}[0]{a_{\max}}
\newcommand{\hamx}[0]{\wh{a}_{\max}}
\newcommand{\bmin}[0]{\beta_{\min}}

\algnewcommand{\algorithmicforeach}{\textbf{for each}}
\algdef{SE}[FOR]{ForEach}{EndForEach}[1]
  {\algorithmicforeach\ #1\ \algorithmicdo}
  {\algorithmicend\ \algorithmicforeach}

\newcommand{\Par}[0]{\mathsf{Par}}
\newcommand{\Anc}[0]{\mathsf{Anc}}
\newcommand{\Desc}[0]{\mathsf{Desc}}

\newcommand{\bfT}[0]{\mathbf{T}}
\newcommand{\PCREMn}[0]{\mathbb Q_{\beta,n}^a}
\newcommand{\PCREM}[0]{\mathbb P^a}

\newcommand{\Qatri}[0]{\mathbb Q^a_{\beta,n\to \iy}}
\newcommand{\Patri}[0]{\mathbb P^a_{n\to \iy}}
\newcommand{\Eatri}[0]{\mathbb E^a_{n\to \iy}}
\newcommand{\Pai}[0]{\mathbb P^a}

\newcommand{\Patr}[0]{\mathbb P_{n\to N}^{a}}

\newcommand{\Qatr}[0]{\mathbb Q_{\beta, n \to N}^{a}}
\newcommand{\Pa}[0]{\mathbb P^{a}_N}
\newcommand{\Ea}[0]{\mathbb E^{a}_N}
\newcommand{\Pan}[1]{\mathbb P^{a}_{#1}}
\newcommand{\Ean}[1]{\mathbb E^{a}_{#1}}
\newcommand{\mun}[1]{\mu_{\be, #1}}
\newcommand{\muv}[1]{\mu_{\be}^{#1}}
\newcommand{\Zn}[1]{Z_{\be, #1}}

\newcommand{\Zhatan}[1]{\wh Z^a_{\be, #1}}
\newcommand{\Zhatn}[1]{\wh Z_{\be, #1}}
\newcommand{\Zhatv}[1]{\wh Z_{\be}^{#1}}
\newcommand{\ZhatnN}[2]{\wh Z_{\be,#1\to #2}}
\newcommand{\ZhatanN}[2]{\wh Z_{\be,#1\to #2}^a}
\newcommand{\Zpn}[1]{Z'_{\be, #1}}
\newcommand{\Ztn}[1]{\wt Z_{\be, #1}}
\newcommand{\Ztnv}[2]{\wt Z_{\be, #1}^{#2}}
\newcommand{\ZtnN}[2]{\wt Z_{\be, #1\to #2}}
\newcommand{\pv}[1]{p_{\be, #1}}
\newcommand{\ptv}[1]{p_{\be, #1}'}
\newcommand{\munv}[2]{\mu_{\be, #1}^{#2}}
\newcommand{\Znv}[2]{Z_{\be, #1}^{#2}}
\newcommand{\Zhatnv}[2]{\wh Z_{\be, #1}^{#2}}
\newcommand{\Zoln}[1]{\overline Z_{\be, #1}}
\newcommand{\muNn}[1]{\mun{N}|_{#1}}
\newcommand{\nun}[1]{\nu_{\be, #1}}
\newcommand{\gap}[0]{\sqrt{\ln 2} - \beta \sfc{\amx}2}

\title{Sampling from the Continuous Random Energy Model in Total Variation Distance}
\author{
\begin{tabular}{c c}
  \begin{tabular}{c}
    Holden Lee\footnote{Department of Applied Mathematics and Statistics, Johns Hopkins University. Email: hlee283@jhu.edu. } 
  \end{tabular} \qquad 
  \begin{tabular}{c}
    Qiang Wu\footnote{School of Mathematics, University of Minnesota. Email: wuq@umn.edu.} 
  \end{tabular}
\end{tabular}}
\date{\today}

\begin{document}
\maketitle

\begin{abstract}
    The continuous random energy model (CREM) is a toy model of spin glasses on $\{0,1\}^N$ that, in the limit, exhibits an infinitely hierarchical correlation structure. 
    We give two polynomial-time algorithms to approximately sample from the Gibbs distribution of the CREM 
    in the high-temperature regime $\beta<\beta_{\min}:=\min\{\beta_c,\beta_G\}$, based on a Markov chain and a sequential sampler. 
    The running time depends algebraically on the desired TV distance and failure probability and exponentially in $(1/g)^{O(1)}$, where $g$ is the gap to a certain inverse temperature threshold $\beta_{\min}$; this contrasts with previous results which only attain $o(N)$ accuracy in KL divergence.
    If the covariance function $A$ of the CREM is concave, the algorithms work up to the critical threshold $\beta_c$, which is the static phase transition point; while for $A$ non-concave, if $\beta_G<\beta_c$, the algorithms work up to the known algorithmic threshold $\beta_G$ proposed in~\cite{addario2020algorithmic} for non-trivial sampling guarantees. 
    Our result depends on quantitative bounds for the fluctuation of the partition function and a new contiguity result of the ``tilted" CREM obtained from sampling, which is of independent interest.
    We also show that the spectral gap is exponentially small with high probability, suggesting that the algebraic dependence is unavoidable with a Markov chain approach.
\end{abstract}

\newpage

\tableofcontents

\newpage

\section{Introduction}

Spin glasses are models of disordered magnetic alloys in statistical physics, which have been extensively studied in the last 50 years. 
To understand the theoretical aspects of spin glasses, Derrida~\cite{Derr81} introduced a simplified spin glass model, the random energy model (REM), where the energy states for each spin configuration are independent random variables. Later, in order to capture more complex features of spin glasses, this toy model was generalized in~\cites{Derr85,derrida1988polymers,BK04, BK04b} by allowing 
hierarchical correlation structure among the energy states. One of these generalizations of REM is the continuous random energy model (CREM), where the energies have continuously many levels of hierarchical correlations. In this paper, we are interested in an algorithmic question for sampling from the CREM Gibbs measure. We first introduce the mathematical definition of the CREM.

\subsection{The continuous random energy model}

For a given positive integer $N$, let $\mathbb T_N:=\bigcup_{n=0}^N\{0,1\}^n$ denote the vertices of a binary tree of depth $N$. Denote the root by $\phi$.  
For $v\in \mathbb T_N$, let $|v|$ denote the length or depth of $v$, that is,  $|v|=n$ if $v\in \{0,1\}^n$. 
For $v, w\in \{0,1\}^n$, let $v\wedge w$ be the greatest common ancestor vertex, which is represented by the longest initial substring appearing in both $v$ and $w$. To help visualize these definitions, Figure~\ref{fig:crem-tree} depicts a particular example of the binary tree with $N=5$.


\begin{df}\label{d:crem}
Fix a positive integer $N$ and let $A:[0,1]\to \R_{\ge 0}$ be a non-decreasing function. We define the \vocab{continuous random energy model (CREM)} with covariance function $A(x)$ as follows. For convenience, we also use $a:[0,N]\to \R_{\ge 0}$ with $a(x) = N\cdot A\pf xN$ to denote the unnormalized covariance function, and define $\widehat{A}$ to be the concave hull of $A$.
\begin{itemize}
\item \ul{Underlying probability space.}
Let $\Omega_N:= \mathbb{R}^{\mathbb{T}_N}$ be the underlying probability space, and equip it with the product Gaussian measure $\Pa$ whose marginal distributions are 
\[\om_u\sim 
\begin{cases}
     \cal N(0, a(0)), &u=\phi,\\
     \cal N \pa{0, a(|u|) - a(|u|-1)}, & u\in \T_N \bs \{\phi\}.
\end{cases}
\] 
We call $\Pa$ the \vocab{CREM disorder measure}.
\item \ul{Random variables.} Define the collection of independent random variables $Y_u(\om) = \om_u$ for $u\in \T_N$.  
For $0\le n\le N$ and $v\in\{0,1\}^n$, define the energy $X_v$ of $v$ by 
\[
X_{v_1\cdots v_n} := \sum_{m=0}^n Y_{v_1\cdots v_m}.
\]
We also define the filtration $\sF_n := \si(\set{Y_u}{|u|\le n})$, i.e.~the $\si$-algebra generated by the random variables up to depth $n$.
\item \ul{Random measures.} 
For a measurable space $S$, let $\mathcal{P}(S)$ denote the space of probability measures on $S$. For $0\le n\le N$, define the distribution associated to the CREM at depth $n$ as the random probability measure $\mun{n}:\Omega_N \to \mathcal{P}(\{0,1\}^n)$ given by\footnote{Note that it would be more proper to write $\mun{n}(\om)(\{v\})$ as $\mun{n}$ is a measure only given a particular $\om\in \Om_N$; we will omit dependence on $\om$.}
    \begin{align*}
\mun{n} (v) &:= \rc{\Zn{n}} e^{\be X_v}
\text{ for each }v\in\{0,1\}^n \\
\text{where }
 \Zn{n} &:= \sum_{v\in \{0,1\}^n} e^{\be X_v}.
    \end{align*}
Above, $\beta>0$ is the inverse temperature parameter, and $\Zn{n}$ is the partition function. When $n =N$, then $\mun{N}$ is the standard \vocab{Gibbs measure} for the CREM; see Remark~\ref{rem:Gibb-def}. 
We write for short $\pv{v} = \mun{n}(v)$. Define the normalized partition function $\Zhatan{n} = \Zhatn{n} = \fc{\Zn{n}}{\Ea \Zn{n}}$, where $\Ea$ is the expectation under $\Pa$, and we omit $a$ when it is clear.
\end{itemize}
\end{df}
One can think of $Y_{v_1\cdots v_m}$ as the random variable assigned to the edge from $v_1\cdots v_{m-1}$ to $v_1\cdots v_m$, so that $X_u$ is the sum of the root label $Y_\phi$ and the random variables along the edges from the root to $u$. Note our definition is slightly more general than in the literature as we allow $A(0)>0$, i.e., $Y_\phi$ to be nontrivial.

\begin{rem}\label{rem:Gibb-def}
In \Cref{d:crem}, if $n=N$, then $\mun{N}$ reduces to the standard definition of  
    the \vocab{Gibbs measure} for the \vocab{continuous random energy model (CREM)} with covariance function $A(x)$. In the standard definition~\cite{BK04b}, for each $u\in \{0,1\}^N$, 
  the energy $(X_u)_{u\in \mathbb T_N}$ is a centered Gaussian process with $\E[X_vX_w] = N\cdot A\pf{|v\wedge w|}{N} = a(|v\wedge w|)$, and it is not necessary to specify the underlying probability space. 
  We explicitly specify the probability space for convenience of change-of-measure arguments in Section~\ref{s:seq}.
\end{rem}



In the special case of CREM with $A(x)=x$, the $X_u$'s form a time-homogeneous branching random walk (BRW), and $\mun{N}$ is the corresponding BRW Gibbs measure. This model has been understood very well in the literature~\cite{shi2016branching}. However, for general $A$, the increments $Y_u$ have non-homogeneous variance depending on $|u|$, which makes the model more challenging to analyze. In addition, there is another similar model, known as the generalized random energy model (GREM)~\cite{BK04}, where instead of branching at every step, there are only a constant number of increments (see \Cref{d:grem}). Finally, as $N\to \infty$, given $a:\R_{\ge 0}\to \R_{\ge 0}$, one can also define an infinite CREM on $\{0,1\}^\N$ (see \Cref{d:crem-iy}).

The algorithmic question of interest is to efficiently sample from the CREM Gibbs measure. More precisely, the goal is to design algorithms with polynomial running time, which output a sample whose distribution is close to the desired distribution under some metric. Before presenting our main results, let us briefly describe some existing algorithmic results for the CREM. A more detailed exposition can be found in Section~\ref{sec:background}.

Sampling from the CREM has drawn much attention mainly due to the paper~\cite{addario2020algorithmic}, where the authors studied the question of determining the algorithmic hardness threshold to understand the relationship between a random energy landscape and algorithmic barriers. For CREM, they conjectured a threshold $\beta_G$ (c.f.~\eqref{eq:beta_g}) 
such that efficient sampling under the Kullback-Leibler (KL) divergence is only tractable for $\beta<\beta_G$. 
A recent work by Ho~\cite{Ho23a} confirms this for the CREM with non-concave $A$ by presenting a polynomial time algorithm in the regime $\beta<\beta_G$. However, their algorithm only gives guarantees of sublinear ($o(N)$) KL divergence. It is not known whether efficient sampling is still possible in total variation (TV) distance up to $\beta_G$. In this paper, we give two new, efficient sampling algorithms which achieve the stronger guarantee of $\ep$ TV distance. We also note that our results hold at high temperature, $\beta<\bmin$, where $\bmin$ is a threshold defined below.  

\begin{figure}
    \center
    \includegraphics[scale=0.5]{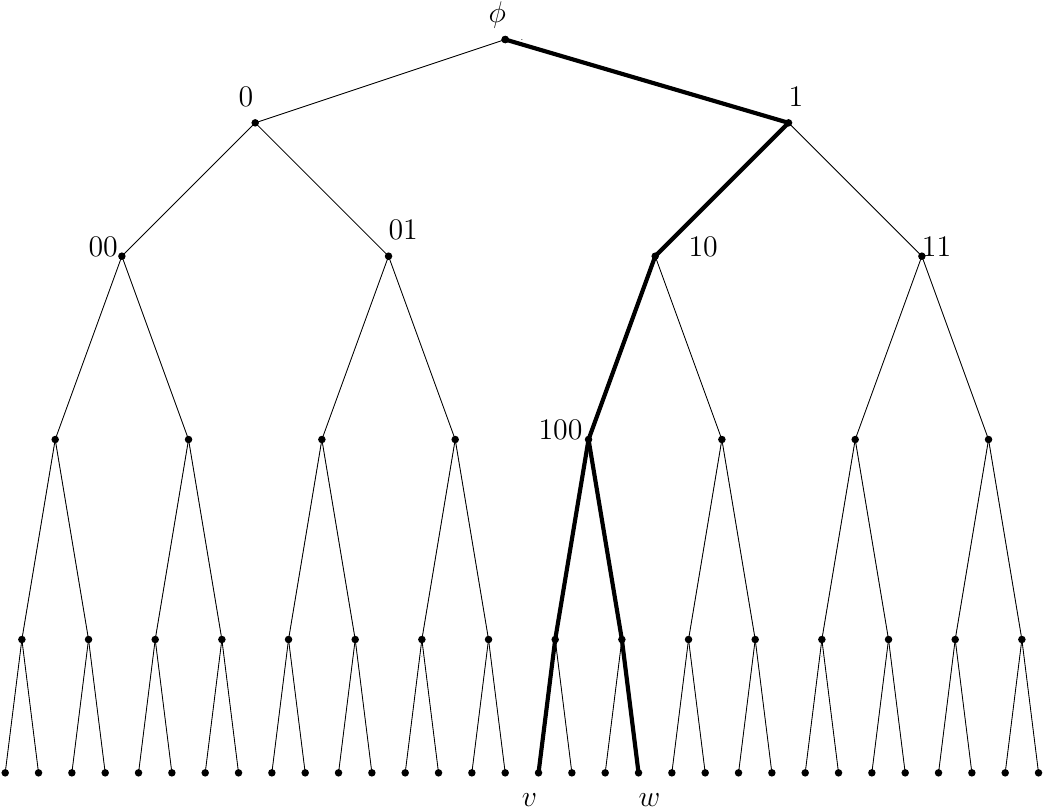}
    \caption{Example of the CREM tree with $N=5$. All vertices are denoted as strings of 0s and 1s. The length of string corresponds to the depth of the vertex. In the above graph, for vertices $v=10000$ and $w=10011$, by definition $v \wedge w=100$ denotes the longest common ancestor vertex. }
    \label{fig:crem-tree}
\end{figure}

\subsection{Main results}

We consider the CREM at high temperature. Below, we use $f'$ to denote the right derivative of $f$.

\begin{assm}[CREM at high temperature]\label{a:hi-temp}
    Suppose that $A:[0,1]\to \R$ is a 
    non-decreasing function with $A(0)=0$ and $A(1)=1$. 
    Consider two assumptions:
    \begin{enumerate}
        \item $\be<\bmin$, where
        \[
\bmin := \sfc{2\ln 2}{\amx}\text{ where }\amx = \sup_{x\in [0,1)} A'(x).
        \]
        Let $g = \gap$ denote the gap.
        \label{i:1}
        \item $\be<\be_c$, where 
        \[\be_c:= \sfc{2\ln 2}{\hamx}\text{ where }\hamx = \sup_{x\in [0,1)} \wh A'(x) = \wh A'(0).\] 
        \label{i:2}
    \end{enumerate}
    We consider the CREM with covariance function $A$. Equivalently, we are given the unnormalized covariance function $a:[0,N]\to \R$ with $a(0)=0$, $a(N)=N$, 
    $\amx = \sup_{x\in [0,N)} a'(x)$, 
    and 
    $\hamx = \sup_{t\in [0,N)}\wh a'(x) = \wh a'(0)$.
\end{assm}
Here, $\be_c$ is the critical inverse temperature \cite{BK04b} where the model undergoes a phase transition between the paramagnetic ($\be<\be_c$) and spin glass ($\be>\be_c$) phase. We also define the critical point for the algorithmic phase transition $\be_G$ as in~\cite{addario2020algorithmic,Ho23a}, 
\begin{align}\label{eq:beta_g}
  \beta_G:=
  \begin{cases}
  \sfc{2\ln 2}{\sup_{t: A(t) \ne \wh A(t)} A'(t)}, & \quad \text{if $A$ non-concave}, \\
  \infty, & \quad \text{if $A$ concave.}
  \end{cases}
\end{align}
We note that 
\[\bmin = \min \{\be_c, \be_G\},\] 
because $\amx = \sup_{x\in [0,1]}A'(x)$ is the supremum over both the set where $A(x)=\wh A(x)$, as in $\be_c$, and the set where $A(x)\ne \wh A(x)$, as in $\be_G$. 
In the special case where 
$A$ is concave, $\bmin=\be_c<\be_G=\iy$. If the sup of the derivative is in the non-concave portion (where $\wh A \ne A$), then $\bmin=\be_G \le  \be_c$.


Our main theorem is that in the high-temperature regime $\be<\bmin$, we can achieve efficient sampling under TV distance either by a Markov chain or a sequential sampling algorithm. Here, efficient means with polynomial dependence on system size $N$, failure probability $\de$, and TV accuracy $\ep$, when $g$ and $\amx$ are fixed. We have not optimized the exponents and they can likely be improved. We allow our algorithms access to intermediate values of the tree, $X_u$ for all $u\in \T_N$.
\begin{thm}\label{t:main-hi-temp}
    Given Assumption~\ref{a:hi-temp}(\ref{i:1}) ($\be<\bmin$), 
    with probability at least $1-\de$ over the CREM, to generate a sample from a distribution
    $\wh \mu$ such that 
    $\TV(\wh\mu,\mun{N})<\ep$, 
\begin{enumerate}
    \item 
    \Cref{a:mc-crem} (Markov chain sampler) takes time $\pf{\amx}{\de g}^{O(\amx/{g}^4)} + \pf{N}{\ep {g}^{1/g}}^{O(1/g)} \ln \prc{\de}$.
    \item 
    \Cref{a:seq} (sequential sampler) takes time $\pf{N\amx}{\ep g}^{O({\amx}^2/{g}^8)}$.
\end{enumerate}
\end{thm}

\begin{rem}
    \Cref{t:main-hi-temp} does not place any assumption on the concavity of $A$. \cite[Theorem 1.13]{Ho23a} shows that for CREM with non-concave $A$, with high probability, sampling under KL divergence in the regime $\beta>\beta_G$ is hard. This readily implies that sampling under TV distance is hard. In the case of non-concave $A$, where $\bmin=\beta_G$, \Cref{t:main-hi-temp}, completes the algorithmic picture for CREM. 
    On the other hand, in general $\bmin<\be_G$ and Ho's algorithm achieving sublinear KL divergence works up to $\be_G$; in particular, for concave $A$ where $\beta_G = \infty$, Ho's algorithm works for all $\beta>0$. It is an open question whether in general, sampling under TV distance also works for all $\beta<\beta_G$, or there is a region where it is possible to sample with sublinear KL but not $\ep$ in TV. 
\end{rem}

\begin{algorithm}[h!]
\caption{Markov chain sampler}
\label{a:mc-crem}
\begin{algorithmic}[1]
\INPUT CREM instance with 
unnormalized covariance function $a$, 
inverse temperature $\be$, and depth $N$; failure probability $\de$; TV accuracy $\ep$
\State Let $m_0 = \ce{C\amx {g}^{-4}\ln \pf{\amx}{g\de}}$ where 
$C$ is an appropriate universal constant. 
\State If $m_0\ge N$, then sample by brute force via computing $e^{\be X_v}$ for all $|v|=N$. Otherwise: 
\State Define a distribution $\pi$ on the tree by 
\begin{align*}
    \pi(v) &\propto 
    \begin{cases}
        \sum_{w\in \Desc^0(v),\,|w|=m_0} e^{\be X_w}, & |v|\le m_0\\
        \fc{\Zn{m_0}}{\Ea [\Zn{m}|\sF_{m_0}]} e^{\be X_v}
        = \Zn{m_0} \cdot 2^{-m} \exp\pa{-\fc{\be^2}2(a(m)-a(m_0)) }\cdot e^{\be X_v},
        & |v|> m_0
    \end{cases}
\end{align*}
\State Run the Markov chain (\Cref{a:mc-tree}) on the tree with stationary distribution $\pi$ for $T$ steps, where $T=\ce{\pf{N}{{g}^{1/g}\de}^{C'/g} \ln \pf N\ep}$ for a universal constant $C'$, to obtain $V_T$.
\OUTPUT $V_T$ if $|V_T|=N$. Otherwise, re-run the algorithm. 
\end{algorithmic}
\end{algorithm}

\begin{algorithm}[h!]
\caption{Markov chain sampler for distribution on tree}
\label{a:mc-tree}
\begin{algorithmic}[1]
\INPUT Distribution $\pi\propto \wt\pi$ on binary tree $\T_N$, number of steps $T$.
\State Let $V_0=\phi$ (root node).
\For{$t=1$ to $T$}
    \State Let $n=|V_{t-1}|$. \algorithmiccomment{current depth}
    \State With probability $\rc 3$ each, propose a transition to the parent or the two children. (Automatically reject if the parent is chosen while at the root, or a child is chosen while at a leaf.)
    \State If the parent $\Par(V_{t-1})$ is chosen, then accept (that is, set $V_t=\Par(V_{t-1})$) with probability \[
\min\bc{
\fc{\wt \pi(\Par(V_{t-1}))}{\wt \pi(V_{t-1})}, 1
}.
    \]
    \State If a child $V_{t-1}x$ is chosen ($x\in \{0,1\}$), then accept (set $V_t= V_{t-1}x$) with probability
    \[
\min\bc{
\fc{\wt \pi(V_{t-1}x)}{\wt\pi(V_{t-1})}, 1
}.
    \]
    \State If we reject, set $V_t = V_{t-1}$.
\EndFor{}
\OUTPUT{$V_t$}
\end{algorithmic}
\end{algorithm}

\begin{algorithm}[h!]
\caption{Sequential sampler}
\label{a:seq}
\begin{algorithmic}[1]
\INPUT CREM instance with unnormalized covariance function $a$, inverse temperature $\be$, and depth $N$; accuracy parameter $\ep$
\State Set $v=\phi$ (root node).
\State Let $m = \ce{C\amx^2 {g}^{-8}\ln \pf{N\amx}{g\ep}}\vee N$ where 
$C$ is an appropriate universal constant. \algorithmiccomment{Lookahead for approximation} 
\For{$t=1$ to $N-m$}
    \State For $x=0, 1$, compute \algorithmiccomment{Currently, $|v|=t-1$.}
    \[\Ztnv{N}{vx} = \Zhatnv{t+m}{vx} = 2^{-m} \exp\pa{-\fc{\be^2}{2} (a(t+m)-a(t))} \sum_{w\in \{0,1\}^m} e^{\be (X_{vxw}-X_{vx})}.\]
    \State Set \algorithmiccomment{Choose a child using estimated marginal probabilities.}
    \begin{align*}
        v\mapsfrom \begin{cases}
            v0, &\text{with probability }\fc{e^{\be Y_{v0}}\Ztnv{N}{v0}}{e^{\be Y_{v0}}\Ztnv{N}{v0} + e^{\be Y_{v1}}\Ztnv{N}{v1}}\\
            v1, &\text{with probability }\fc{ e^{\be Y_{v1}}\Ztnv{N}{v1}}{e^{\be Y_{v0}}\Ztnv{N}{v0} + e^{\be Y_{v1}}\Ztnv{N}{v1}}.
        \end{cases}
    \end{align*}
\EndFor{}
\State Sample $w\in \{0,1\}^m$ with probability $p(w)\propto e^{\be (X_{vw}-X_v)}$. \Comment{Currently, $|v|=N-m$.}
\OUTPUT $vw$.
\end{algorithmic}
\end{algorithm}

For an explanation of the notations used in Algorithms~\ref{a:mc-crem} and~\ref{a:seq}, see \Cref{s:notation}. The Markov chain algorithm relies on a polynomial bound on the $s$-conductance---that is, on showing that sets of size at least $s$ (chosen to depend on $\ep, \de, N$) are not too bottlenecked in the Markov chain. We note that the algebraic rate probably cannot be improved to a geometric rate because the conductance and spectral gap are exponentially small. For simplicity, we show this in the case that $A(x)=x$ (the branching random walk). 
\begin{thm}[Upper bound for spectral gap]\label{t:ub-gap}
Consider the CREM with $A(x)=x$ and inverse temperature $\be>0$. 
With probability $1-e^{-\Om(N)}$ over the randomness in the CREM, the spectral gap of the Markov chain in \Cref{a:mc-crem} is $\exp(-\Om(N))$, where the implicit constants depend on $\be$. 
\end{thm}
We briefly describe the two algorithms. For the Markov chain \Cref{a:mc-crem}, we simulate a Markov chain on the tree $\T_N$ based on a Metropolis-Hastings step (\Cref{a:mc-tree}), such that the stationary distribution restricted to level $n$ is $\mun{n}$. For this to efficiently give a sample at the $N$th level, we must ensure that the probabilities assigned to each level of the tree are comparable, which suggests that it needs to divide the unnormalized probability $e^{\beta X_v}$ at level $m$ by $\Ea \Zn{m}$. To take care of fluctuations for small $m$, we in fact first compute exactly the partition function at a level $m_0$ and condition on its value. (This is not necessary, but makes the bounds nicer.)

We remark that as stated, \Cref{a:mc-crem} will only output a valid sample $\Te(1/N)$ of the time, and this can be made more efficient by reweighting the $N$th level to have constant probability under the stationary distribution. For simplicity, we analyze the algorithm without reweighting.

The sequential sampling \Cref{a:seq} relies on the following result for approximating the full partition function, using the partition function at some depth that is logarithmic in the desired error. 
\begin{thm}\label{t:main-Z}
There is a constant $C$ such that 
if $m\ge C\pa{\fc{\amx}{g^4} \ln \prc{\de g} + \rc{g^{2}} \ln \prc \ep}$, then
given Assumption~\ref{a:hi-temp}(\ref{i:1}) ($\be<\bmin$), 
    \[
    \Pa\pa{\ab{\fc{\Zhatn{m}}{\Zhatn{N}}- 1}\le \ep}\ge 1-\de.  
    \]
\end{thm}
If one can approximately compute the partition functions for the conditional distributions after restricting the initial $n$ coordinates for any $n$, then it is possible to approximately sample via a sequential procedure: Once we have sampled the first $t-1$ coordinates $v$, we approximately compute the partition functions for $v0$ and $v1$ to then sample the next coordinate with the appropriate probability. (Note that \Cref{a:hi-temp}(\ref{i:1}) implies \Cref{a:hi-temp}(\ref{i:1}) for the CREM starting from any level $n$, so \Cref{t:main-Z} can be applied to the subtrees.)
The problem with this reduction is that \Cref{t:main-Z} only holds with high probability under the randomness of the CREM, and even if the first $n$ coordinates are sampled perfectly according to the CREM, the subtree we arrive at is no longer a CREM, but a tilted one, because vertices whose subtrees have larger partition functions are more likely to be chosen. However, we can bound the Radon-Nikodym derivative between the tilted measure and the CREM, and hence still obtain the result of \Cref{t:main-Z} under the tilted measure with a larger failure probability. We note that the same argument shows a result on contiguity of these tilts for the infinite CREM, which may be of independent interest (see \Cref{t:contiguous}).

\subsection{Background and related works}\label{sec:background}

For the early toy models for spin glasses (REM, GREM and CREM), not many positive results on efficient sampling exist. For REM, it was shown in~\cite{MR1627811} that the spectral gap is exponentially small, and thus the Metropolis dynamics is slow mixing. Similar spectral gap estimates have recently been obtained in the GREM case~\cite{NF20}.\footnote{We also show an exponentially small spectral gap for the CREM in Theorem~\ref{t:ub-gap}, so we note with caution that small spectral gap does not necessarily preclude MCMC algorithms from giving efficient algorithms.} 
However, one can hope that the continuous hierarchical structure of the CREM can give algorithms more of a ``foothold" in finding low-energy states. Indeed, there has been recent positive progress for sampling from the CREM.

Addario-Berry and Maillard~\cite{addario2020algorithmic} investigated the algorithmic threshold problem for optimization. They showed that there exists a threshold $x_*$ in terms of $A$ in the following sense. For any $\epsilon>0$, a linear time algorithm can output a vertex $v\in \{0,1\}^N$ with $X_v \ge (x_* - \epsilon)N$, while no polynomial time algorithm can find $v$ such that $X_v\ge (x_*+\epsilon)N$. Guided by this result, they raised the problem of sampling from the CREM and conjectured an explicit sampling hardness threshold $\beta_G$. 
They conjecture that for $\beta< \beta_G$, there exists an efficient algorithm that outputs a sample $v\in \{0,1\}^N$ close to a the true distribution under the Kullback-Leibler (KL) divergence, while for $\beta>\beta_G$, no polynomial time algorithm can achieve this.

Later, Ho and Maillard~\cite{Ho22} gave an efficient recursive algorithm for a special case of CREM, the branching random walk. Recently, Ho~\cite{Ho23a} extended these to the general case of CREM with a similar recursive algorithm and proved the conjecture. More concretely, for $A$ non-concave, they proved that in the algorithmically tractable regime ($\beta<\beta_G$), their algorithm can output a sample in polynomial time that close to a real sample with expected KL divergence scaling as $o(N)$. In the low-temperature regime ($\beta>\beta_G$), they showed that with high probability, any algorithm achieving the above sampling task will take at least exponential amount of time. For $A$ concave, efficient sampling is achievable for all $\beta>0$.


Beyond the above toy models, the sampling problem for more general spin glass models has recently seen much progress; in particular, the Sherrington-Kirkpatrick (SK) model has been the subject of intense study. The SK model is a mean-field model with random pairwise interactions between spins: the energy for each configuration $\sigma \in \{-1,+1\}^N$ is 
\[
H_N(\sigma) := \frac{1}{\sqrt{N}}\sum_{i<j} J_{ij} \sigma_i \sigma_j,
\]
and the associated SK Gibbs measure is accordingly defined as 
\[
\mu_{\text{SK}}(
\sigma) := \exp(\beta H_N(\sigma)) \cdot \left(\sum_{\sigma\in \{-1,+1\}^N} \exp(\beta H_N(\sigma))\right)^{-1}.
\]
It is known that the SK model undergoes a phase transition at $\beta=1$. In the high temperature regime ($\beta<1$), the model is replica symmetric, where the spin configurations are asymptotically independent. In the low temperature regime ($\beta >1$), the SK model has been conjectured to exhibit the so-called full-step replica symmetry breaking (RSB) phase. More specifically, it means that the Gibbs measure is asymptotically supported on an ultrametric tree with continuously many branchings, which is similar to the CREM hierarchical correlation structure. In this sense, the CREM can be treated as an approximation of the SK model. 

For sampling from the SK model, in 2019, Bauerschmidt and Bodineau~\cite{BB19} first proved a (non-standard) log-Sobolev inequality 
associated with the SK Gibbs measure. Later Eldan, Koehler and Zeitouni~\cite{EKZ22}, using the technique of stochastic localization, established a spectral gap estimate for the standard Glauber dynamics and thus derived a fast mixing result for Glauber dynamics. This automatically gives an efficient sampling algorithm (in TV distance). However, the above two results hold at a sub-region ($\beta<1/4$) of the entire high temperature regime ($\beta<1$). Extending the fast mixing of Glauber dynamics for SK to $\beta<1$ is still an open question. Similar results on proving functional inequalities for the more general $p$-spin glass models have been obtained in~\cites{ABXY24,AJKPV24}. Note that those spectral gap results are in sharp contrast with the results for the REM-based models. On the other hand, El Alaoui, Montanari and Sellke~\cite{EMS22} took a different approach by discretizing the stochastic localization process to sample from SK model, which, combined with the work of \cite{celentano2024sudakov},  works for $\beta<1$.  However, the sampling algorithm is only guaranteed under the Wasserstein-2 metric, which is weaker than the total variation distance as in~\cite{EKZ22}. Similar results have been obtained for more general $p$-spin models~\cite{EMS23}. Additionally, Huang, Montanari and Pham~\cite{HMP24} recently obtained an efficient sampling algorithm for spherical $p$-spin glass models under total variation distance, where the spin values are not discrete but on a continuous sphere. The algorithm is still based on discretizing the localization process.

To summarize, due to the critical differences on spectral gap estimates, although the REM models act as simplification for the SK model and its variants, designing and proving efficient sampling algorithms for the CREM is not necessarily easier. Besides, for the Ising spin glass models, efficient sampling can either only be achieved at high enough temperature under total variance distance, or under some weaker metric notion up to the critical threshold. Similarly for the CREM, 
Ho's result~\cite{Ho23a} can efficiently sample up to the hardness threshold $\beta_G$, but only gives weak KL divergence guarantees. Arguably, our result is the first efficient sampling algorithm for discrete spin glasses that (in some cases) works up to the critical threshold under the TV distance.

\subsection{Notation}\label{s:notation}

For a vertex of a binary tree $v=v_1\cdots v_m\in \T_N$, define the parent, ancestors, and descendants as follows, and the set version (for $S\subeq \T_N$) in the natural way.
\begin{align*}
\Par(v) &= v_1\cdots v_{m-1} & \Par(S) &= \set{\Par(v)}{v\in S}\\
\Desc(v) &= \set{v_1\cdots v_{m}x\in \T_N}{x \in \bigcup_{n>0} \{0,1\}^n} & \Desc(S) &= \bigcup_{v\in S} \Desc(v) \\
\Anc(v) &= \set{v_1\cdots v_n}{n<m} & 
\Anc(S) &= \bigcup_{v\in S} \Anc(v)
\end{align*}
Also define $\Desc^0$ and $\Par^0$ to include the vertex itself. We also use $\T^v_N$ to denote the tree rooted at $v$ ($\T^v_N = \Desc^0(v)$).

We also define the partition function starting at $v$ ($|v|=n$) with depth $m$, and the normalized partition function by 
\begin{align}
\label{e:Znv}
\Znv{n+m}{v} &= \sum_{|u|=m} e^{\be (X_{vu}-X_v)}, \\
\label{e:Zhatnv}
\Zhatnv{n+m}{a,v} &= \fc{\Znv{n+m}{v}}{\Ean{N} \Znv{n+m}{v}} = \fc{\Znv{n+m}{v}}{2^m e^{\fc{\be^2}2 (a(n+m)-a(n))}}.
\end{align}
We will write $\Zhatnv{n+m}{v} = \Zhatnv{n+m}{a,v}$ when $a$ is understood. Note that we use the subscript to $Z$ to denote the total depth in the original tree, rather than the depth starting from $v$, as in some previous works. 

We defer further definitions we will need until \Cref{s:seq}; see 
\Cref{s:nomen} for a complete list of notations for probability measures, random measures, and partition functions.

\subsection{Structure of the paper}

In \Cref{s:fe}, we show (quantitative) concentration properties of the partition function $\Zn{N}$ that are of general interest and used to prove guarantees for both the Markov chain and sequential sampling algorithms. This allows us to prove efficient approximation of the partition function under the CREM (\Cref{t:main-Z}). 
In \Cref{s:mc}, we prove the main theorem (\Cref{t:main-hi-temp}) using the Markov chain approach, and in \Cref{s:spectral} we prove the upper bound on the spectral gap (\Cref{t:ub-gap}). 
In \Cref{s:seq}, we combine a contiguity argument with \Cref{t:main-Z} to prove \Cref{t:main-hi-temp} using sequential sampling.

\section{Free energy and partition function approximation}
\label{s:fe}
We derive concentration properties of $Z_{\be, N}$ under the CREM, which will be essential for both approaches (Markov chain and sequential sampler). The bounds in this section allow us to prove \Cref{t:main-Z}, that the normalized partition function for depth $m$ approximates the full (depth $N$) partition function. In \Cref{s:Z-add}, we show this in an additive sense, i.e. we bound $|\Zhatn{m} - \Zhatn{N}|$. By then deriving a high-probability lower bound for $\Zhatn{N}$, we can then prove a multiplicative approximation (as in \Cref{t:main-Z}) in \Cref{s:Z-mult}.

We first recall the fundamental result 
\cite[Theorem 3.3]{BK04b}, which shows that fixing $A$ (with $A(0)=0$), the limiting free energy as $N\to \iy$ is 
\begin{align}
F_\be:=
\lim_{N\to \iy} \rc N \E_N \ln Z_{\be, N} 
&= \int_0^1 f\pa{\be \sqrt{\wh A'(s)}} \,ds,
\label{e:FE1}
\end{align}
where the convergence is in probability and
\[
f(x) = \begin{cases}
    \ln 2 + \fc{x^2}{2} , & x<\sqrt{2\ln 2}\\
    \sqrt{2\ln 2}\cdot x, & x\ge \sqrt{2\ln 2}.
\end{cases}
\]
(See \cite{Ho23a} for the extension to Riemann integrable $A$.)
In particular, if $\be<\be_c$, then $F_\be = \ln 2 + \fc{\be^2A(1)}{2}$. 
\cite[Theorem 3.1]{BK04b} computes the expected ground state energy density, 
\begin{align}
\lim_{N\to \iy}\rc N  
\E \max_{v\in \{0,1\}^N} \be X_{v} \to \be \sqrt{2\ln 2}\int_0^1 \sqrt{\wh A'(s)}\,ds.
\label{e:lim-max}
\end{align}
For our application, we will need to track $\Zn{m}$ for $0\le m\le N$, and so consider \eqref{e:FE1} and \eqref{e:lim-max} with integrals $\int_0^x$, $0\le x\le 1$. We note that if $\be \ge \be_c$, then for small $x$, \eqref{e:FE1} and \eqref{e:lim-max} will be equal.
If instead $\be< \be_c$, there will be a gap for any $x$, and hence we can expect the largest probability mass $p_{\be,v}$ at level $n$ to be exponentially small in $n$. We will take advantage of this to show fast convergence for $\Zhatn{n} = \fc{\Zn{n}}{\E_N \Zn{n}}$: we consider the martingale increments and observe that a sum of terms that each contribute an exponentially small amount to the total will be very concentrated. 

To work with the martingale $\Zhatn{n}$, we first note the calculation
\begin{align}
\label{e:EZ}
\Ea \Zn{n} &= 2^n \E_{X\sim \cal N(0,a(n))} e^{\be X} 
= 2^n \exp\pa{\fc{\be^2a(n)}2} 
= 2^n \exp\pa{\fc{\be^2N \cdot A\pf{n}{N}}2}\\
\label{e:FE}
\implies  \rc N \ln \Ea \Zn{N}& = \ln 2 + \fc{\be^2A(1)}2.
\end{align}
This quantity \eqref{e:FE} is called the annealed free energy. (To distinguish, \eqref{e:FE1} is called the quenched free energy.) 
Rephrasing the result \eqref{e:FE1}, when $\be<\be_c$, the quenched and annealed free energies are equal in the limit. 
We need to make quantitative the concentration of $\Zn{N}$ around $\E \Zn{N}$. 
For this, first note we have concentration of $\rc N \ln \Zn{N}$ around its mean $\rc N \E \ln \Zn{N}$, and then derive quantitative bounds for the approximation $ \rc N \E\ln \Zn{N} \approx \rc N \ln \E \Zn{N}$. 

Finally, we note that while we only assume $\be<\be_c$ in the above discussion, we will also need $\be<\be_G$, as we need the results to hold for the CREM starting at any intermediate level (see \Cref{l:renorm}).

\label{s:Z}
\subsection{Fluctuation of free energy and additive approximation}
\label{s:Z-add}

In this section, we 
bound $|\Zhatn{m} - \Zhatn{N}|$ using martingale arguments. We bound the expected maximum value of $X_v$ over leaves $v\in \{0,1\}^N$. While we expect any particular $X_v$ to be on the order of $\sqrt N$, the fact that there are $2^N$ vertices means that we expect the maximum to be on the order of $N$. 
\eqref{e:lim-max} gives the limiting expected maximum; to obtain concentration we consider the expectation of $\max e^{\be X_v}$. For our purpose, it suffices to give an upper bound where the coefficient of $N$ in the exponent is strictly less than $F_\be$. 
\begin{lem}\label{l:max}
For the CREM with covariance function $A$ (and unnormalized covariance function $a$), when $\be \le \sfc{2\ln 2}{A(1)}$, we have
    \[
    \Ea \max_{|v|=N} e^{\be X_v} 
    \le  2e^{\be N \sqrt{(2\ln 2) A(1)}}
    = 2e^{\be \sqrt{N(2\ln 2) a(N)}}.
    \]
\end{lem}
\begin{proof}
By considering $\wt A(x) = \fc{A(x)}{A(1)}$ and $\wt \be = \sqrt{A(1)}\cdot \be$, it suffices to prove the lemma when $A(1)=1$.  
By Sudakov-Fernique, we obtain an upper bound by replacing the $X$'s by independent random Gaussians with the same variance,
\begin{align*}
    \Ea \max_{|v|=N} e^{\be X_v} 
    &\le 
    \Ea \max_{i\in [2^N]} e^{\be \sqrt{N} Y_i} 
\end{align*}
where $Y_i\sim \cal N(0,1)$ are iid. Now, using a union bound and the Gaussian tail bound in \Cref{l:gtail} \eqref{e:gaussian-tail-lower}, $\Pj_{Y\sim \cal N(0,1)}(Y\ge u) \le e^{-u^2/2}$, 
\begin{align*}
    \Ea \max_{i\in [2^N]} e^{\be \sqrt{N} Y_i} 
    &\le \int_0^\iy \Pa\pa{\max_{i\in [2^N]} e^{\be \sqrt{N} Y_i}  \ge u}\,du\\
    &\le \int_0^\iy \min\bc{1, 2^N \Pj_{Y\sim \cal N(0,1)} \pa{Y\ge \fc{\ln u}{\be \sqrt N}}}\,du\\
    &\le  e^{\be N \sqrt{2\ln 2}}
    + 2^N\int_{e^{\be N \sqrt{2\ln 2}}}^\iy  \Pj_{Y\sim \cal N(0,1)} \pa{Y\ge \fc{\ln u}{\be \sqrt N}}\,du\\
    &=  e^{\be N \sqrt{2\ln 2}}
    + 2^N\int_{\be N \sqrt{2\ln 2}}^\iy  \Pj_{Y\sim \cal N(0,1)} \pa{Y\ge \fc{v}{\be \sqrt N}}e^v\,dv\\
    &\le e^{\be N \sqrt{2\ln 2}}
    + 2^N \int_{\be N \sqrt{2\ln 2}}^\iy 
    \fc{\be \sqrt{N}}{\sqrt{2\pi}v} e^{-\fc{v^2}{2\be^2N}} 
    e^v\,dv\\
    &\le  e^{\be N \sqrt{2\ln 2}}
    + 2^N e^{\rc 2 \be^2N} \int_{\be N \sqrt{2\ln 2}}^\iy 
    \fc{1}{\sqrt{2\pi N\cdot 2\ln 2}} e^{-\rc{2\be^2 N} \pa{v - \be^2 N}^2} 
    \,dv\\
    &\le  e^{\be N \sqrt{2\ln 2}}
    + 2^N e^{\rc 2 \be^2N} \fc{\be}{\sqrt{2\ln 2}} \Pj_{Y\sim \cal N(0,1)}
    \pa{Y\ge \sqrt N (\sqrt{2\ln 2}-\be)} \\
    &\le 
     e^{\be N \sqrt{2\ln 2}}
    + 2^N e^{\rc 2 \be^2N} \fc{\be}{\sqrt{2\ln 2}} e^{-\fc{N (\sqrt{2\ln 2}-\be)^2}{2}}\\
    &\le 
     e^{\be N \sqrt{2\ln 2}}
    + e^{N\pa{\ln 2 + \fc{\be^2}2-\fc{ (\sqrt{2\ln 2}-\be)^2}{2}}}
    = 2 e^{\be N \sqrt{2\ln 2}}
\end{align*}
when $\be \le \sqrt{2\ln 2}$. 
\end{proof}

A basic result is that $\Zhatn{n}$ is a martingale. For the case of branching random walks, convergence properties of this martingale (termed the additive martingale) has been well-studied \cite{biggins1977martingale,biggins1992uniform,shi2016branching}.
\begin{lem}\label{l:mg}
Under $\Pa$, $(\Zhatn{n})_{0\le n\le N}$ is a martingale adapted to $(\sF_n)_{0\le n\le N}$.
\end{lem}
\begin{proof}
    Let $\E = \Ea$. We calculate, because $(Y_w)_{|w|>n}$ are independent of $\sF_n$, that
    \begin{align*}
        \E[Z_{\be, n+1} | \sF_n]
        &= \E\ba{\sum_{v\in \{0,1\}^n} \sum_{x\in \{0,1\}} e^{\be (X_v + Y_{vx})} | \sF_n} \\
        &= \sum_{v\in \{0,1\}^n} e^{\be X_v} \sum_{x\in \{0,1\}} \E\ba{e^{\be Y_{vx}} | \sF_n} \\
        &= \sum_{v\in \{0,1\}^n} e^{\be X_v} \cdot 2\E_{Y\sim \cal N(0,a(n+1)-a(n))}[e^{\be Y}]\\
        &= \Zn{n} \cdot 2\E_{Y\sim \cal N(0,a(n+1)-a(n))}[e^{\be Y}]\\
        \E[Z_{\be, n+1}] & =
        \E \ba{\E[Z_{\be, n+1} | \sF_n]}\\
        &= \E \Zn{n} \cdot 2\E_{Y\sim \cal N(0,a(n+1)-a(n))}[e^{\be Y}].
    \end{align*}
    Hence
    \begin{align*}
        \E[\wh Z_{\be, n+1} | \sF_n]
        &=\fc{\E[Z_{\be, n+1} | \sF_n]}{\E[Z_{\be, n+1}]} = 
        \fc{\Zn{n}}{\E[\Zn{n}]} = \Zhatn{n}.
    \end{align*}
\end{proof}

For the following two lemmas, we let $R(x) := \fc{A(x)}{x}$, note that $R(x)\le \sup \wh A'=\hamx$, and define
\begin{align*}
    \ga_1 &= \be \sqrt{(2\ln 2)\hamx}, &
    \ga_2 &= \ln 2 + \fc{\hamx \be^2}2,\\
    \ga_1(n) &= \be \sqrt{(2\ln 2)R(n/N)}, & \ga_2(n) &= \ln 2 + \fc{R(n/N)\be^2}2.
\end{align*}
From \eqref{e:EZ}, we obtain $\Ea \Zn{n} = e^{\ga_2(n)n}$. 

 Using Markov's inequality, we show that there is small probability such that the maximum of $X_v$, $|v|=n$ is large.
\begin{lem}
\label{l:stop-large}
Let $\ep>0$ and fix $m\ge 1$. 
    Define the stopping time 
    \begin{align*}
        S_m &= \min\set{n\ge m}{\max_{|v|=n} e^{\be X_v} > e^{(\ga_1(n)+\ep)n}}.
    \end{align*}
    Then 
    \[
\Pa(S_m\le N) \le \fc{2}{\ep}e^{-\ep(m-1)}.
    \]
\end{lem}
\begin{proof}
We have
\begin{align}
\nonumber
    \Pa(S_m\le N) & = 
    \Pa\pa{\bigcup_{n=m}^N 
    \{
    \max_{|v|=n} e^{\be X_v} > e^{(\ga_1(n) + \ep)n}\}
    }\\
    \nonumber
    &\le \sum_{n=m}^{\iy} 
    \Pa\pa{\max_{|v|=n} e^{\be X_v} > e^{(\ga_1+\ep) n}}\\
    &\le \sum_{n=m}^{\iy} \fc{\Ea \max_{|v|=n} e^{\be X_v} }{e^{(\ga_1(n)+\ep) n}} 
    &\text{by Markov's inequality}\nonumber\\
    &\le \fc{2e^{\be n \sqrt{(2\ln 2)R\pf nN}}}{e^{(\ga_1(n)+\ep) n}}  & \text{by Lemma~\ref{l:max}}
    \nonumber\\
    &\le \fc{2e^{\ga_1(n) n}}{e^{(\ga_1(n)+\ep) n}} = \sum_{n=m}^{\iy} 
    2e^{-\ep n} \le 
    \fc{2}{\ep} e^{-\ep (m-1)}.
    \label{e:approx-1}
\end{align}
\end{proof}

The following result tells us that we can approximate the partition function of the CREM at level $N$ by the partition function at level $m$, where the error decreases exponentially in $m$.
\begin{lem}\label{l:ZN}
    Given \Cref{a:hi-temp}(\ref{i:2}) ($\be<\be_c$), 
    suppose $0<\de < \fc{\ga_2-\ga_1}3$. Let $\ep = \fc{(\ga_2-\ga_1)-3\de}4$. Then 
        \[
\Pa(|\Zhatn{m} - \Zhatn{N}|\ge e^{-\de m})\le 
\fc{1}{\ep} (2+51e^{\ep}) e^{-\ep(m-1)} + 2e^{-2\ep m}\le \fc{100}{\ep}e^{-\ep m}.
    \]
\end{lem}
\begin{proof}
Let $\E = \Ea$ and $\Pj = \Pa$. 
Because $R(x) \le \hamx$, we have
$\ga_2(n)-\ga_1(n) = \rc 2 \pa{\be \sqrt{R\pf nN} - \sqrt{2\ln 2}}^2 \le  \ga_2-\ga_1$. 
Define the stopping times 
\begin{align*}
    S_m &= \min\set{n\ge m}{\max_{|v|=n} e^{\be X_v} > e^{(\ga_1(n)+\ep)n}},\\
    T_m &= \min\set{n\ge m}{\Zhatn{n} <\rc 2 e^{-(\de + 2\ep)n}}
    = \min\set{n\ge m}{\Zn{n} < \rc 2e^{(\ga_2(n)-(\de + 2\ep))n}}
\end{align*}
and consider $\Zpn{n} = \Zhatn{n\wedge S_m\wedge T_m}$, i.e., we stop the process when one of two bad events happen: either the largest element is too large, or the sum is too small. (Set the times to be $\iy$ if the events have not happened at time $N$.) We choose the thresholds so that if the process is not stopped, the largest element is still exponentially smaller than the sum. More precisely, for $|v|=n$, recall $\pv{v} = \fc{e^{\be X_v}}{\Zn{n}}$, and note that because $\Zhatn{n}$ is a martingale adapted to $\sF_n$ by~\Cref{l:mg}, then 
\begin{align*}
    \E[\Zpn{n+1} - \Zpn{n}|\sF_n] &= 0, 
\end{align*}
and 
\begin{align*}
\Zpn{n+1} - \Zpn{n} & = \sum_{v\in \{0,1\}^n} \pv{v} \pa{\fc{\sum_{x\in \{0,1\}} e^{\be Y_{vx}}}{2e^{\fc{\be^2}2 (a(n+1)-a(n))}}-1} \one_{S_m\wedge T_m>n}, \\
\Var[\Zpn{n+1} - \Zpn{n} | \sF_n]
& = \sum_{v\in \{0,1\}^n} \pv{v}^2 \Var\pa{\fc{\sum_{x\in \{0,1\}} e^{\be Y_{vx}}}{2e^{\fc{\be^2}2 (a(n+1)-a(n))}}-1}  \one_{S_m\wedge T_m>n}\\
&\le \max_{|v|=n} \pv{v}  \cdot \Var_{Y_1,Y_2\sim \cal N(0, a(n+1)-a(n))} \pf{e^{\be Y_1}+e^{\be Y_2}}2 e^{-\be^2(a(n+1)-a(n))}\\
&\le \fc{e^{(\ga_1(n)+\ep)n}}{e^{\ga_2(n) - (\de+2\ep)n}} 
\cdot \rc 2 \fc{e^{2\be^2(a(n+1)-a(n))} - e^{\be^2(a(n+1)-a(n))}}{e^{\be^2(a(n+1)-a(n))} }\\
&\le e^{-[(\ga_2-\ga_1) - \de - 3\ep]n} 
\pa{e^{\be^2(a(n+1)-a(n))}-1}.
\end{align*}
where $(\ga_2-\ga_1) - \de - 3\ep>0$ by assumption. Then 
\begin{multline}
    \Pj\pa{\ab{\Zhatn{m} - \Zhatn{N}}> e^{-\de m}}
    \le \Pj(S_m\le N)  +\Pj \pa{\ab{\Zpn{m} - \Zpn{N}}> e^{-\de m}}\\
    + \Pj\pa{\bc{S_m>N }
    \cap\{T_m\le N\} \cap 
     \bc{\ab{\Zhatn{m} - \Zhatn{N}}> e^{-\de m}}}. 
     \label{e:approx}
\end{multline}
We bound each of the terms. 
The first term is bounded by~\Cref{l:stop-large}.
Next, by Markov's inequality, 
\begin{align}
\nonumber
    &\Pj\pa{\ab{\Zpn{m} - \Zpn{N}}> e^{-\de m}}\\
    \nonumber
    &\le \fc{\E (\Zpn{N} - \Zpn{m})^2}{e^{-2\de m}}\\
    \nonumber
    &\le \fc{\sum_{n=m}^{N-1} e^{-[(\ga_2-\ga_1) - \de - 3\ep] n} 
    \pa{e^{\be^2(a(n+1)-a(n))}-1}
    }{e^{-2\de m}}\\
    \nonumber
    &\le 
     e^{-[(\ga_2-\ga_1)-3\de - 3\ep](m-1)} \fc{e^{(\ga_2-\ga_1)-\de-3\ep}(e^{2\ln 2} - 1)}{(\ga_2-\ga_1)-\de -3\ep}\\
    &\le 
    3 e^{-[(\ga_2-\ga_1)-3\de - 3\ep](m-1)} \fc{e^{(\ga_2-\ga_1)-\de-3\ep}}{(\ga_2-\ga_1)-\de -3\ep}
    \le \fc{3}{\ep} e^{-\ep(m-1)}e^{\ep}.
    \label{e:approx-2}
\end{align}
Finally, for the last term, under that event we note that either $\Zhatn{m}$ or $\Zhatn{N}$ has to be not too small:
\begin{align}
\nonumber
    &\Pj\pa{\bc{S_m>N } \cap\{T_m\le N\} \cap  \bc{\ab{\Zhatn{m} - \Zhatn{N}}> e^{-\de m}}}\\
    &\le \Pj\pa{\bc{S_m>N } \cap\{T_m\le N\} \cap \bc{\Zhatn{m} > e^{-\de m}}} + \Pj\pa{\{T_m\le N\} \cap \bc{\wh Z_{\be, N} > e^{-\de m}}}.
    \label{e:approx-3}
\end{align}
For the first term in~\eqref{e:approx-3}, by the $L^2$ maximum inequality, 
\begin{align}
\nonumber &\Pj\pa{\bc{S_m>N } \cap\{T_m\le N\} \cap \bc{\Zhatn{m} > e^{-\de m}}}\\
\nonumber
 &\le \E\ba{\one_{\Zhatn{m} > e^{-\de m}} 
 \Pj\pa{ \bc{S_m>N } \cap \{T_m\le N\} | \sF_m}}\\
 \nonumber
  &\le \E\ba{
 \Pj\pa{ \max_{n\le m\le N} |\Zpn{m} - \Zpn{N}| \ge e^{-\de m} - \rc 2 e^{-(\de + 2\ep)m} | \sF_m}} \\
 &
 \le \fc{\E ( \max_{m\le n\le N}(\Zpn{n} - \Zpn{m}))^2}{e^{-2\de m}}
 \le
 \fc{4\E (\Zpn{N} - \Zpn{m})^2}{e^{-2\de m}}
 \le 16\cdot \eqref{e:approx-2}.
\end{align}
For the second term in~\eqref{e:approx-3}, we argue that once $\Zhatn{n}$ hits a small value, it has small probability of become large again. Using the fact that $\Zhatn{n}$ is a martingale and Markov's inequality,
\begin{align}
\nonumber
\Pj\pa{\{T_m\le N\} \cap \bc{\Zhatn{N} > e^{-\de m}}}
    &\le \E \ba{\one_{T_m\le N} \cdot \Pj\pa{\Zhatn{N} > e^{-\de m}| \sF_{N \wedge T_m}}} \\
    \nonumber
    &\le \E \ba{\sum_{n=m}^N \one_{T_m=n} 
    \Pj\pa{\Zhatn{N} > e^{-\de m} | \sF_n}}\\
    \nonumber
    &\le \E \ba{\sum_{n=m}^N \one_{T_m=n} 
    \Pj\pa{\fc{\Zhatn{N}}{\Zhatn{n}}> \fc{e^{-\de m}}{\rc 2e^{-(\de+2\ep)n}} | \sF_n}}\\
    &\le 2e^{-2\ep m}.
    \label{e:approx-3b}
\end{align}
Hence~\eqref{e:approx} is bounded by
\begin{align*}
    \Pj\pa{\ab{\Zhatn{m} - \Zhatn{N}}> e^{-\de m}}
    & \le \eqref{e:approx-1} + 17\cdot \eqref{e:approx-2} + \eqref{e:approx-3b}.
\end{align*}
Plugging in the choice of $\ep$ gives the result. The last inequality uses $\ep \le \fc{\ln 2}4$.
\end{proof}

We actually need to bound the ratio $\Zhatn{N} / \Zhatn{m}$. For this we need the following concentration bound, which will upper bound the probability of $\Zhatn{m}$ being exponentially small.
\begin{lem}\label{l:conc-lnZ}
    Given a CREM with covariance function $A$ (or unnormalized covariance function $a$), we have
    \[
\Pa\pa{
\ab{\ln \Zn{N} - \E \ln \Zn{N}}\ge x
}\le 2e^{-\fc{x^2}{4N \cdot A(1)}} = 2e^{-\fc{x^2}{4a(n)}}.
    \]
\end{lem}
\begin{proof}
    This follows from \cite[Theorem 1.2]{panchenko2013sherrington}.
\end{proof}




\subsection{Free energy computation and multiplicative approximation}
\label{s:Z-mult}
\Cref{l:conc-lnZ} gives us concentration of $\Zn{N}$ around $e^{\E \ln \Zn{N}}$, rather than $\E \Zn{N}$. Hence we need to approximate the \emph{free energy} $\rc N \E \ln \Zn{N}$. As mentioned, it is known that $e^{\E \ln \Zn{N}} = \E \Zn{N} + o(N)$ at high temperature; in this section we derive a quantitative error bound.

For this, we approximate the CREM with a model with a fewer number of hierarchies, the generalized random energy model.
\begin{df}\label{d:grem}
The \vocab{generalized random energy model (GREM)} with initial energy $a_0$, lengths $s_1,\ldots, s_n\in \N$, energies $a_1,\ldots, a_n$, and inverse temperature $\be$ as the CREM with unnormalized covariance function 
\[
a(x) = a_0N + \sum_{m=1}^n a_ms_m\one_{x\ge \sumo im s_i}.
\]
\end{df}
For a more explicit definition, let $B_i = \{0,1\}^{s_i}$, $N=\sum_{i=1}^n s_i$. Let $Y_\phi\sim \cal N(0, a_0N)$ where $\phi$ represents the empty string. For each $v_1\in B_1, \ldots, v_i\in B_i$, let $Y_{v_1\cdots v_i}\sim \cal N(0,s_ia_i)$ independently, and
for $v=v_1\cdots v_n \in \prodo in B_i = \{0,1\}^N$, 
let $X_{v_1\ldots v_n} = \sumz mn Y_{v_1\ldots v_m}$. 
Then $\mun{N}$ is a probability distribution on $\{0,1\}^N$ defined by
    \begin{align*}
\mu_{\be,N} (v) &= \rc{Z_{\be,N}} e^{\be X_v}
\text{ for each }v\in\{0,1\}^N\\
\text{where }
 \Zn{N} &= \sum_{v\in \{0,1\}^N} e^{\be X_v}.
    \end{align*} 
Typically, we fix $n$ and consider  $s_i=r_iN$ for fixed $r_i$ with $\sumo in r_i=1$, and take $N\to \iy$.

The following is a quantitative version of Proposition 3.4 in \cite{capocaccia1987existence}, which is proved therein by the second moment method. We lower bound the partition function of GREM with good probability by considering the probability that $Y_{v_i}$ is large for each segment.
\begin{lem}[{\cite[Proposition 3.4]{capocaccia1987existence}}]\label{l:grem-box}
    Consider a generalized random energy model with lengths $s_1,\ldots, s_n$ and energies $a_1,\ldots, a_n$. Then for subsets $\{\De_i\}_{i=1}^n \subseteq \mathbb{R}^n$ defined in~\eqref{eq:Delta},
    \[
\Pj\pa{
\ab{\set{v=v_1\cdots v_n}{\forall i, Y_{v_i}\in \De_i}}
\le (1-\eta) 2^N
\prodo in \Pj(\xi_{i}\in \De_i)
}
\le \rc{\eta^2}
\sumo jn \rc{\prodo ij 2^{s_i} \Pj(\xi_{i}\in \De_i)}
    \]
    where $\xi_i\sim \cal N(0,a_is_i)$ has the distribution of each $Y_{v_i}$.
\end{lem}
In the following, denote by $\Phi$ the cumulative distribution function of the standard normal.
\begin{cor}\label{c:grem-lb}
Consider a generalized random energy model with initial energy $a_0$, lengths $s_1,\ldots, s_{n}\ge 1$ energies $a_1,\ldots, a_{n-1},0$, and inverse temperature $\be < \sfc{2\ln 2}{\max a_i}$. Define $g_i=\sqrt{\ln 2} - \be \sfc{a_i}2$ and suppose $\de\le \min g_i \wedge \rc 2$.
There is a constant $C$ and $c$ such that if
\begin{align}\label{e:conditions-s-N}
s_i \ge \fc{C}{\de}\ln \prc \de, \quad 
N\ge \fc{C}{\de},
\end{align}
then 
\[
\E\ln \Zn{N}
\ge N\ln 2 + \fc{\be^2}2 \sumo i{n-1} a_i s_i - N\de + \be \Phi^{-1}(c\de)\sqrt{Na_0}. 
\]
\end{cor}
To interpret this, note $N\ln 2 + \fc{\be^2}{2} \sumo i{n-1} a_i s_i  = N\pa{\ln 2 + \fc{\be^2}2 (A(1)-A(0))}$; the rest of the terms are error terms. For the CREM with $A(0)=0$, the limiting free energy in the high-temperature regime is $\ln 2 + \fc{\be^2}2A(1) = \rc N \ln \E \Zn{N}$, so after approximating a CREM by a GREM, and in light of \Cref{l:conc-lnZ}, this will help us show concentration of the free energy around $\ln 2 + \fc{\be^2}2A(1)$.
\begin{proof}
Consider $1\le i\le n-1$. 
    By Lemma~\ref{l:eat-blnt}, 
    choosing $s_i \ge \fc{3}{\de} \ln \prc{\de} \vee \fc{2}{\de} \ln (12\be \sqrt a_i)$ implies
    \begin{align}
    \label{e:s2}
\fc{e^{\de s_i}}{\sqrt{s_i}} &\ge 12\be \sqrt{a_i}&
    \implies &&
\rc{12\be \sqrt{a_is_i}} &\ge e^{-\de s_i}.
    \end{align}
    and choosing 
    $s_i \ge \fc{6}{2\ln 2 - \be^2 a_i}\ln \pf{2}{2\ln 2 - \be^2 a_i} \vee \fc{4}{2\ln 2 - \be^2 a_i}\ln \pf{12\be \sqrt{a_i}}{\de}$
    gives
    \begin{align}
    \label{e:s1}
    \fc{e^{(\ln 2 - \rc 2 \be^2 a_i)s_i}}{\sqrt{s_i}}
    & \ge \fc{12\be\sqrt{a_i}}{\de}&
    \implies &&
2^{s_i} \cdot \rc{12\be \sqrt{a_is_i}} e^{-\rc 2 \be^2 a_is_i} &\ge \rc \de.
    \end{align}
    Noting that $\be\sqrt{a_i}\le \sqrt{2\ln 2}$ and $\de \le g_i$, these constraints are implied by $s_i\ge \fc{C}{\de}\ln \prc\de$ for appropriate constant $C$.
    Let  
    \begin{align}\label{eq:Delta}
    \De_i = [\be a_is_i, \iy) \quad \text{for $1\le i<n$} \quad \text{and} \ \De_n=\R. 
    \end{align}
    
    Then by \Cref{l:gtail} and \eqref{e:s2}, recalling $\xi_i\sim \cal N(0,s_ia_i)$,
    \begin{align}
    \nonumber
(1-\eta) 2^N 
\prodo in \Pj(\xi_{i}\in \De_i)
&\ge 
(1-\eta) 2^N \prodo i{n-1} \rc{12\be \sqrt{a_is_i}} e^{-\rc 2 \be^2 a_is_i} \\
\nonumber
&\ge 
(1-\eta) 2^N\prodo i{n-1} e^{-\de s_i-\rc 2 \be^2 a_is_i}\\ 
\nonumber
&\ge 
(1-\eta)e^{N\ln 2 -\rc 2 \be^2 \sumo i{n-1} a_is_i - \de \sumo i{n-1} s_i}\\
&\ge 
(1-\eta)e^{N(\ln 2-\de) -\rc 2 \be^2 \sumo i{n-1} a_is_i}.
\label{e:lots-largeish}
    \end{align}
By \Cref{l:gtail} and \eqref{e:s1}, for $j\le n-1$,  
\begin{align*}
\prodo ij 2^{s_i} \Pj(\xi_{i}\in \De_i)
&\ge \prodo ij 2^{s_i} \rc{12\be \sqrt{a_is_i}} e^{-\rc 2 \be^2 a_is_i}\ge \rc{\de^j}.
\end{align*}
Thus
    \begin{align}
    \label{e:bad-prob-small}
\rc{\eta^2}
\sumo j{n} \rc{\prodo ij 2^{s_i} \Pj(\xi_{i}\in \De_i)}&\le 
\rc{\eta^2}
\pa{\sumo j{n-1} \de^j + \rc2 \de^{n-1}} \le \rc{\eta^2}\cdot 2\de.
\end{align}
Let $G = \set{v=v_1\cdots v_n}{\forall 1\le i\le n,\,  Y_{v_i}\in \De_i}$.
By \Cref{l:grem-box} and \eqref{e:bad-prob-small},
\begin{align}
    \Pj\pa{
\ab{G}
\ge (1-\eta) 2^N
\prod_{i=1}^n \Pj(\xi_{i}\in \De_i)
}
\ge 1-\fc{2\de}{\eta^2}.
\end{align}
Hence letting $\De_0 = [\Phi^{-1}\pa{\de}, \iy)$, 
\begin{align*}
    \Pj\pa{
\ab{G}
\ge (1-\eta) 2^N
\prod_{i=1}^n \Pj(\xi_{i}\in \De_i)
\text{ and }Y_\phi \in \De_0
}
\ge 1-\fc{2\de}{\eta^2} - \de .
\end{align*}
Note that if $Y_\phi\in \De_0$ and $Y_{v_i}\in \De_i$ for all $1\le i\le n$, then 
\begin{align}
    X_{v_1\cdots v_n} &\ge e^{\be \Phi^{-1}(\de)\sqrt{Na_0} + \be^2 \sum_{i=1}^{n-1} a_is_i}.
\end{align}
Hence, by summing over all vertices in $G$ and using~\eqref{e:lots-largeish}, 
\begin{align*}
    \Pj\pa{\Zn{N} \ge  (1-\eta) e^{N(\ln 2-\de) +\rc 2 \be^2 \sumo i{n-1} a_is_i + \be \Phi^{-1}(\de)\sqrt{Na_0} }}
    \ge 1-\fc{2\de}{\eta^2} - \de .
\end{align*}
Hence
\begin{align*}
    \E \ln \Zn{N} &\ge
    \Pj\pa{\Zn{N} \ge  (1-\eta) e^{N(\ln 2-\de) +\fc{\be^2}2 \sumo i{n-1} a_is_i + \be \Phi^{-1}(\de)\sqrt{Na_0} }}\\
    &\quad 
    \cdot 
    \pa{\ln (1-\eta) + N(\ln 2-\de) +\fc{\be^2}2 \sumo i{n-1} a_is_i + \be \Phi^{-1}(\de)\sqrt{Na_0} }.
\end{align*}
Choosing $\eta=\rc 2$, 
\begin{align*}
    \E \ln \Zn{N} & \ge
    (1-9\de) \pa{-\ln 2 + N(\ln 2-\de) + \fc{\be^2}{2} \sumo i{n-1} a_is_i + \be\Phi^{-1}(\de)\sqrt{Na_0}}.
\end{align*}
Using $N\ge \fc C\de$ (for large enough $C$), we obtain
\[
\E \ln \Zn{N} \ge N\ln 2 + \fc{\be^2}2 \sumo i{n-1}a_is_i - O(\de N)  + \be\Phi^{-1}(\de)\sqrt{Na_0}.
\]
Modifying the constant in~\eqref{e:conditions-s-N} as appropriate then finishes the proof.
\end{proof}



We now show a comparison result for the free energy of CREM models.
\begin{lem}[cf. {\cite[Theorem 3.3]{BK04b}}]
\label{l:comp-Z}
    Suppose $A(1)=B(1)$ and $A\le B$. Let $a,b$ be the corresponding unnormalized covariance functions. 
    Then
    \begin{equation}\label{e:comp-Z}
\E_N^a \ln \Zn{N} \ge \E_N^b \ln \Zn{N}.
    \end{equation}
    In particular, if $s_1,\ldots, s_n\in \N$, 
    $s_1+\cdots +s_n=N$ and we define
    $a_0 = A\pf{s_1}{N}$, 
    $a_i = \fc N{s_i}\pa{ A\pf{s_1+\cdots +s_{i+1}}{N} - A\pf{s_1+\cdots +s_{i}}N}$ for $1\le i\le n-1$, then  
    letting $B(x)=a_0N + \sum_{i=1}^N \one_{x\ge \fc{s_1+\cdots +s_{i}}{N}} a_i$, we have that \eqref{e:comp-Z} holds and $\Zn{N}^B$ is the GREM with initial energy $a_0$, lengths $s_1,\ldots, s_n$ and energies $a_1,\ldots, a_{n-1},0$. 
    Moreover, if $\be< \sfc{2\ln 2}{\sup A'}$ and $s_1\ge s_2\ge \cdots \ge s_n$, then 
    $\be < \sfc{2\ln 2}{\max_i a_i}$ (so 
    the bound in Corollary~\ref{c:grem-lb} holds).
    \end{lem}
Note that taking an upper approximation to $A$ requires a ``staggered" version of the GREM.
\begin{proof}
    Note that for $v,w\in \T_N$, 
    \begin{align*}
        \Ea [X_v X_w] = N \cdot A\pf{|v\wedge w|}{N}
        \le N \cdot B\pf{|v\wedge w|}{N} = 
        \E_N^b [X_v X_w].
    \end{align*}
    The lemma then follows from Slepian's inequality.

    For the last statement, note that the fact that $(s_i)$ is decreasing means that for $1\le i\le n-1$, 
    \begin{align*}
a_i& = \fc N{s_i}\pa{ A\pf{s_1+\cdots +s_{i+1}}{N} - A\pf{s_1+\cdots +s_{i}}N}\\
&\le  \fc N{s_{i+1}}\pa{ A\pf{s_1+\cdots +s_{i+1}}{N} - A\pf{s_1+\cdots +s_{i}}N} \le \sup A',
    \end{align*}
    so $\be<\sfc{2\ln 2}{\sup A'} \le \sfc{2\ln 2 }{\max a_i}$. 
\end{proof}

We can now derive a bound for the error between $\Ea \ln \Zn{N}$ and $\ln \Ea  
\Zn{N}$ by comparison to a GREM and using the lower bound for the free energy of the GREM (\Cref{c:grem-lb}). 
\begin{lem}\label{l:Eln-lnE}
Consider a CREM with covariance function $A$ satisfying \Cref{a:hi-temp}(\ref{i:1}) ($\be<\bmin$). 
Suppose $0<\de < \sqrt{\ln 2} - \be \sfc{\amx}2$.
    There is a constant $C$ such that for $N\ge \fc{C \amx}{\de^2}\ln \prc \de$, 
    \[
\Ea \ln \Zn{N} \ge N\pa{\ln 2 +\fc{\be^2A(1)}2 - \de} = -\de N + \ln \Ea  
\Zn{N}.
    \]
\end{lem}
\begin{proof}
    Choosing $s_1=\cdots =s_{n-1}=\ce{\fc{N}{n}}$, 
    and defining 
    $a_i = \rc{s_i}\pa{ A\pf{s_1+\cdots +s_{i+1}}{N} - A\pf{s_1+\cdots +s_{i}}N}$
    and $B$ be as in \Cref{l:comp-Z},
we have that
\[
\sumo i{n-1} a_i s_i \ge N\pa{A(1) - A\pf{s_1}{N}} 
\ge N\cdot A(1) - \amx \fc{2N}{n},
\]
and $\Ea \ln  \Zn{N} \ge \E_N^b \ln \Zn{N}$. By \Cref{c:grem-lb}, if $s_i \ge \fc{C}{\de} \ln \prc{\de}$ and $N\ge \fc{C}{\de}$, then 
\begin{align*}
\E_N^b \ln \Zn{N} 
&\ge N\ln 2 + \fc{\be^2}2 \sumo i{n-1} a_i s_i - N\de + \be \Phi^{-1}(c\de)\sqrt{Na_0}\\
&\ge 
N\pa{\ln 2 + \fc{\be^2A(1)}2 - \de} - \fc{\be^2 N}2\cdot A\pf{s_1}N + \be \Phi^{-1}(c\de)\sqrt{Na_0}\\
&\ge N\pa{\ln 2 + \fc{\be^2A(1)}2 - \de} - \fc{\be^2 N \amx}n + \be \Phi^{-1}(c\de)\sqrt{Na_0}
\end{align*}
We now aim to ensure that the conditions of \Cref{c:grem-lb} are satisfied, and the last two terms above are at most $\de N$:
\begin{enumerate}
    \item The conditions of \Cref{c:grem-lb} are satisfied when $N \ge \fc{Cn}{\de}\ln \prc{\de}$ and $N\ge \fc{C}{\de}$. We used the fact that $s_i \ge \fc Nn$. 
    \item (First term) $\fc{\be^2 N \amx}{n}\le \de N$: Using the fact that $\fc{\be^2\amx}2\le \ln 2$, this is satisfied when 
    $n\ge \fc{2\ln 2}{\de}$. 
    \item (Second term) 
    $ \be \Phi^{-1}(c\de)\sqrt{Na_0}\le \de N$: Using $\Phi^{-1}(c\de)\asymp \sqrt{\ln \prc {c\de}}$ and $a_0 = A\pf{s_1}{N}\le \fc{2\amx}{n}$, $\be \sqrt{\amx}\le \sqrt{2\ln 2}$, 
    it suffices for 
    \[
Nn \ge \fc{C\ln \prc \de}{\de^2}
    \]
    for an appropriate constant $C$.
\end{enumerate}
For appropriate constants, when $N\ge \fc{C_1}{\de^2}\ln \prc \de$, we can choose $n\sim \fc{C_2}{\de}$ to make these conditions all satisfied. Replacing $\de$ with $\de/3$ and using \eqref{e:EZ} then gives the lemma.
\end{proof}

Combining this with concentration of the free energy (\Cref{l:conc-lnZ}) and the additive concentration result (\Cref{l:ZN}), we are now able to show a multiplicative concentration result and prove the result on partition function approximation (\Cref{t:main-Z}).
\begin{lem}\label{l:ratio-1} \label{c:max-pv}
\begin{enumerate}
    \item There are constants $c_1,c_2$ such that the following hold. 
Given \Cref{a:hi-temp}(\ref{i:1}) ($\be<\bmin$), 
if $m \ge \fc{C\amx}{g^4}\ln \prc g$ for an appropriate constant $C$, then
    \[
\Pa\pa{\ab{\fc{\Zhatn{m}}{\Zhatn{N}}-1}\ge e^{-c_1 g^2m}}\le e^{-\fc{c_2g^4m}{\amx}}.
    \]    
    \item 
    There are constants $C, c_1,c_2$ such that the following hold. 
    Given \Cref{a:hi-temp}(\ref{i:1}) ($\be<\bmin$), 
    if $m \ge \fc{C\amx}{g^4}\ln \pf{\amx}{g}$, then 
\[
\Pa\pa{\max_{|v|\ge m} \pv{v} \le e^{-c_1g^2|v|}} \ge 1-e^{-\fc{c_2 g^4 m}{\amx}}.
\]
\end{enumerate}
\end{lem}
\begin{proof}
Suppose $\de= c
g^2$ for a sufficiently small $c$.

    For the first part, we apply~\Cref{l:conc-lnZ} with $m$ in place of $N$ and 
    the same $a$. 
    Note that $\fc{(\de m)^2}{4 a(m)} \ge \fc{\de^2m}{\amx}$. 
    We obtain that with probability $\ge 1-2e^{-\fc{\de^2 m}{4\amx}}$ that 
    \begin{align}\label{e:lnZ-error}
\ln \Zn{m} \ge \Ea \ln \Zn{m} - \de m \ge \ln \Ea \Zn{m} - 2\de m
    \end{align}
    by \Cref{l:Eln-lnE} when $m\ge \fc{C\amx}{\de^2}\ln \prc \de$. This gives $\Zn{m}\ge e^{-2\de m}\Ea \Zn{m}$, or $\Zhatn{m} \ge e^{-2\de m}$. 
    By \Cref{l:ZN}, noting $\ga_2-\ga_1 = g^2$,  we then obtain that 
    \[
\Pa\pa{\ab{1-\fc{\Zhatn{N}}{\Zhatn{m}}}\ge e^{-c_1 g^2m}}\le
e^{-\fc{c_2g^4m}{\amx}}
    \]
    for some constants $c_1,c_2$. Adjusting constants as necessary, the bound for the reciprocal also holds. This shows the first part.

    For the second part, from~\Cref{l:stop-large}, for $\ga_1(n) n = \be \sqrt{(2\ln 2) Nn \cdot A\pf nN}$,  
    \[
\Pa\pa{\forall m\le n\le N, \,
\max_{|v|=n} e^{\be X_v} \le e^{(\ga_1(n) + \ep)n}} \ge 1-\fc 2\ep e^{-\ep (m-1)}.
    \]
    Also, with probability $\ge 1-2e^{-\fc{\de^2 m}{4\amx}}$, by \eqref{e:lnZ-error},
\[
\Zn{m} \ge e^{-2\de m} \Ea \Zn{m}
= e^{-2\de m + m\ln 2 + \fc{\be^2 N \cdot A\pf mN}2} 
\]
With probability $\ge 1-\fc{8\amx}{\de^2} e^{-\fc{\de^2(m-1)}{4\amx}}$, this holds for all $n$ such that $m\le n\le N$. Under these two events, dividing gives
\[
\max_{|v|=n} \pv{v} \le 
\exp\ba{-\pa{\sqrt{n\ln 2} - \sfc{\be^2 N \cdot A\pf nN}2}^2 +\ep n +2\de n}
\le \exp\pa{-ng^2+(\ep+2\de)n}.
\]
Choosing $\ep, \de = cg^2$ for a small enough constant $c$ then gives the result.
\end{proof}


\begin{proof}[Proof of \Cref{t:main-Z}]
    In \Cref{l:ratio-1}, it suffices to choose the constant in the statement of \Cref{t:main-Z} to ensure that \[m\ge \max\bc{
\fc{C\amx}{g^4}\ln \prc g , \fc{\amx}{c_2g^4} \ln \prc \de, \rc{c_1g^2} \ln \prc \ep
    }.\]
\end{proof}

We note for later that we can apply the results in this section for $\Zn{n\to N}$ for all $0\le n< N$, exactly when \Cref{a:hi-temp}(\ref{i:1}) ($\be<\bmin$)
is satisfied. Informally, \Cref{a:hi-temp}(\ref{i:1}) 
for the whole tree in the CREM implies 
\Cref{a:hi-temp}(\ref{i:1}) and \emph{a fortiori} (\ref{i:2}) ($\be<\be_c$) holds for every subtree. Also note that even if we only required $\be<\be_c$ for every subtree, we would still have to assume \Cref{a:hi-temp}(\ref{i:1}) ($\be<\bmin$) for the original CREM, because the CREM starting from depth $xN$ will have its $\be_c\le \sfc{2\ln 2}{A'(x)}$.
\begin{lem}\label{l:renorm}
    If \Cref{a:hi-temp}(\ref{i:1}) ($\be<\bmin$)
    holds for the CREM with depth $N$, covariance function $A$, and inverse temperature $\be$, then for all $0\le n<N$, 
    \Cref{a:hi-temp}(\ref{i:1})  
    holds for the CREM with depth $N^*= N-n$, covariance function $A^*(x)=\fc{A\pa{\fc{n}{N} + x \fc{N-n}{N}} - A\pf{n}{N}}{A(1) - A\pf nN}$ and inverse temperature $\be^* = \be \sqrt{\fc{N}{N-n} \pa{A(1) - A\pf nN}}$.
    Thus, under the CREM with covariance function $A$, all results for $\Zn{N}$ which hold under \Cref{a:hi-temp}(\ref{i:1})  
    with some dependence on (a lower bound for) $g$, would also hold for $\Zn{n\to N}$ under \Cref{a:hi-temp}(\ref{i:1}) 
    with the same dependence on $g$.
\end{lem}
\begin{proof}
    Our choice of $N^*$, $A^*(x)$, and $\be^*$ ensures for $0\le m\le N-n$ that 
    \[
    \be^{*2} N^* A^*\pf{m}{N^*} = \be^2 N \pa{A\pf{m+n}{N} - A\pf nN}.
    \]
    We need to check \Cref{a:hi-temp}(\ref{i:1}) 
    for this renormalized model, 
    \begin{align*}
        \be^{*} &< \sfc{2\ln 2}{\sup {A^*}'}\\
    \iff 
        \be \sqrt{\fc{N}{N-n}\pa{A(1) - A\pf nN}}
        &< \sfc{2\ln 2}{\sup_{[\fc nN, 1]} A'}\sqrt{\fc{N}{N-n}\pa{A(1) - A\pf nN}}.
    \end{align*}
    This is implied by \Cref{a:hi-temp}(\ref{i:1}). 
    Moreover, the gap for the renormalized model is
    \begin{align*}
        \sqrt{\ln 2} - \be^* \sfc{\sup {A^*}'}2 =
        \sqrt{\ln 2} - \sqrt{
        \fc{\sup_{[\fc nN, 1]}A'}2}\ge g.
    \end{align*}
\end{proof}


\section{Markov chain sampler}
\label{s:mc}

To prove our main result (Theorem~\ref{t:main-hi-temp}) using the Markov chain sampler, we show convergence of the Markov chain from a warm start using $s$-conductance. 
We first give some general results on the $s$-conductance in Section \ref{s:conductance}. 
The $s$-conductance of a Markov chain is roughly the worst boundary-to-volume (i.e. ``bottleneck") ratio a set could have under the stationary distribution, when restricted to volumes in $(s, \rc 2]$. In the special case of trees, we show that it suffices to consider unions of (complete) subtrees as our sets.

For the CREM, to bound the $s$-conductance, we will trim the subtrees of vertices $v$ such that $\wh Z_{\be, N-|v|}^v \ge 
\fc{N}{\ep} \Ea \wh Z_{\be, N-|v|}^v$. 
In Section \ref{s:tail}, we bound the $p$th moment of $\Zhatn{n}$ to show that on average, 
this trimming operation cuts out a small amount of mass (some function of $\ep$, which we make to be $<s$) from the leaves. The remainder of the tree will then have good conductance.

Finally, we turn these expectation computations into a probability result to prove our main theorem in Section~\ref{s:conductance-crem}.

We note that the proof using $s$-conductance gives $\ep$-mixing times from a warm start that are powers of $\rc{\ep}$. In Section~\ref{s:spectral}, we suggest that this is unavoidable: with high probability, the spectral gap (and hence conductance) for CREM is exponentially small, and hence the worst-case mixing time is exponentially large.

\subsection{Conductance on trees}

\label{s:conductance}
One common strategy to show good mixing for Markov chains is to lower bound the $s$-conductance, or bottleneck ratio. In the following, we use $\mu_k$ for $k\ge 0$ to denote the distribution of Markov chain at $k$-th step.

\begin{df}
For $s\in [0,\rc2)$, define the \vocab{$s$-conductance} $\Phi_s$ of a Markov chain with transition kernel $T$ and stationary distribution $\pi$ to be 
    \[
\Phi_s := 
\inf_{A:\pi(A)\in (s,1-s)}
\fc{\int_A T(u,A^c)\,\pi(du)}{\min\{\pi(A)-s, \pi(A^c)-s\}}.
    \]
    Define the \vocab{conductance} to be $\Phi:=\Phi_0$.
\end{df}

A bound on $s$-conductance gives a bound on convergence of a Markov chain in TV distance up to an additive constant, when initialized at a warm start.
\begin{lem}
\label{l:conductance-mixing}
Consider a Markov chain on $\Om$ with stationary distribution $\pi$ and transition kernel $T$, and let $\mu_k = \mu_0 T^k$. 
Let $\gamma\ge 2$ and 
    suppose that $\mu_0$ is a $\gamma$-warm start, that is, the Radon–Nikodym derivative $\dd{\mu_0}{\pi}\le \gamma$. Suppose that $\pi$ has no atoms of size $\ge \fc 18$. 
Then 
    \[
    \TV(\mu_k, \pi) \le 
    (\gamma-1) \pa{s + \rc 2\left( 1- \frac{\Phi_s^2}{8}\right)^k}
    \le (\gamma-1) \pa{s + \rc 2 e^{-\fc{k\Phi_s^2}8}}.
    \]
\end{lem}
This lemma is similar to~{\cite[Corollary 1.5]{lovasz1993random}} except that in our case, the upper bound on the RHS is in terms of $s$-conductance. The proof idea is essentially the same.
\begin{proof}
    First, for any integer $k\ge 0$, we let 
    \[
    h_k(x):= \sup_{g:\Om \to [0,1] \,:\,\int_\Om g \,d\mu  =x } \int_\Om g \,(d \mu_k - d \mu).
    \]
    For $g$ satisfying the constraints, we have
    \begin{align*}
        \int_\Om g \,(d\mu_k - d\mu) &\le 
        \int_\Om g \pa{\dd{\mu_k}{\mu} - 1}\,d\mu 
        \le \int g (\gamma-1)\,d\mu = (\gamma-1)x\\
        \int_\Om g \,(d\mu_k - d\mu) &=
        \int_\Om (1-g) \pa{1-\dd{\mu_k}{\mu}}\,d\mu \le 1-x.
    \end{align*}
    For $\gamma\ge 2$ and $s\le \rc \gamma$, we have
    \begin{align*}
        h_0(x) 
        &\le \min \{(\gamma-1) x, 1-x\} \\
        &\le (\gamma - 1)s + (\gamma - 1)\sqrt{\rc \gamma - s} \min\bc{\sqrt{x-s}, \sqrt{1-x-s}}. 
    \end{align*}
    (The second term is taken to be 0 if one of the terms is undefined.)
    To see the second inequality, note that it holds for $x\in [0,s]\cup [1-s,1]$, and for $x\in [s,1-s]$, the RHS is concave. Since the LHS is the union of two line segments intersecting at a point in $[s,1-s]$, it suffices to check the inequality holds at the maximum of the LHS, $x=\rc \be$. In this case, we have equality.
    
    Now applying \cite[Theorem 1.4]{lovasz1993random}\footnote{The original statement of Theorem 1.4 in~\cite{lovasz1993random} has upper bound in terms of $\left( 1-\Phi_s/2\right)^k$. The proof relies on \cite[Lemma 1.3]{lovasz1993random}, where the probability space was assumed atom-free. In our case with atoms of size less than 1/8, one can apply \cite[Lemma 1.3*]{lovasz1993random} instead to obtain a bound with $\left( 1-\Phi_s/8\right)^k$.} gives
    \begin{align*}
    h_k(x) \le 
    (\gamma - 1)s + (\gamma - 1)\sqrt{\rc \gamma - s} \min\bc{\sqrt{x-s}, \sqrt{1-x-s}}
    \left( 1- \frac{\Phi_s^2}{8}\right)^k
    \end{align*}
    and
    \begin{align*}
    \TV(\mu_k, \pi) 
    &\le \sup_{x\in [0,1]}h_k(x)
    \le (\gamma-1) \pa{s + \sqrt{\pa{\rc \gamma - s} \pa{\rc 2-s}}\left( 1- \frac{\Phi_s^2}{8}\right)^k}\\
    &\le (\gamma-1) \pa{s + \rc 2\left( 1- \frac{\Phi_s^2}{8}\right)^k}.
    \end{align*}
\end{proof}



We first derive general results for the Markov chain on the tree described by \Cref{a:mc-tree}.
Recall that $\T^v_N$ or $\Desc^0(v)$ denotes the tree rooted at $v$. The following lemma says that for a Markov chain on a (binary) tree, it suffices to consider the conductance of subtrees.
\begin{lem}\label{l:trees-suffice}
For the Markov chain in \Cref{a:mc-tree}, 
we have
\begin{align}
\label{e:phis-subtrees}
\Phi_s \ge \minr{A:\pi (A)\ge s}{A=\Desc^0(S)} \rc 3\cdot  \fc{\pi(S)}{\pi(A)-s},
\end{align}
where the minimum is over all sets $A$ with $\pi(A)\ge s$ that are disjoint unions of subtrees.
\end{lem}
Note that we don't constrain $\pi(A)\le \rc 2$ on the RHS. Note in \eqref{e:phis-subtrees}, it suffices to consider when $S=\set{v\in A}{\Par(v)\nin A}$. 
\begin{proof}
Let $Q (A, B) = \sum_{v\in A, w\in B} \pi(v) T(v, w)$ where $T$ is the transition matrix; this represents the flow between $A$ and $B$ under the stationary distribution.  
Note that because the Markov chain is reversible, $Q(A,A^c)=Q(A^c,A)$ and we have 
\[
\Phi_s = \inf_{A:\pi(A)\in (s, \rc 2]} \fc{Q(A,A^c)}{\pi(A)-s}.
\]
It suffices to lower bound $\fc{Q(B,B^c)}{\pi(B)}$ for any $B$ with $\pi(B)\in (s, \rc 2]$.
Consider two cases.

\paragraph{Case 1.} $\phi\nin B$ (The root is not in $B$). Let $\ol B = \Desc^0(B)$. Note that $Q(B,B^c) \ge Q(\ol B, \ol B^c)$ because the set of ``exposed" vertices only shrinks.
Then $s\le \pi(B)\le \pi(\ol B)$ and 
\[
\Phi_s(B) = \fc{Q(B,B^c)}{\pi(B)-s}
\ge \ub{\fc{Q(\ol B,\ol B^c)}{\pi(\ol B)-s}}{ \Phi_s(\ol B)}
= \rc3 \fc{\sum_{v\in S} \min \{\pi(v), \pi(\Par(v))\}}{\pi(\ol B)-s}.
\]
Let $S=\set{v\in \ol B}{\Par(v)\nin \ol B}$; it suffices to consider the RHS of~\eqref{e:phis-subtrees} for such $S$. 
Now let
\begin{align*}
    S_1 &= \set{v\in S}{\pi(v)\le  \pi(\Par(v))} & 
    S_2 &= \set{v\in S}{\pi(v)> \pi(\Par(v))}
\end{align*}
so that 
\begin{align*}
    \Phi_s(\ol B) &
    = \rc 3
    \fc{\sum_{v\in S} \min \{\pi(v), \pi(\Par(v))\}}{\pi (\ol B) - s} 
    = \rc 3
    \fc{\sum_{v\in S_1} \pi(v) + \sum_{v\in S_2}\pi(\Par(v))}{\sum_{v\in S} \pi(\T^v_N)-s}\\
    &\ge \rc 3
    \fc{\sum_{v\in S_1 \bs \Desc(\Par(S_2))} \pi(v) + \sum_{v\in S_2}\pi(\Par(v))}{\sum_{v\in S_1 \bs \Desc(\Par(S_2))} \pi(\T^v_N) + \sum_{v\in S_2} \pi (\T^{\Par(v)}_N)-s}
\end{align*}
In the denominator, for the vertices in $S_1$ that are descendants of $\Par(v)$, $v\in S_2$, their measure is accounted for by the fact that we expanded the tree from $v$ to $\Par(v)$. Then, letting $A = \bigcup_{v\in S_2} \T^{\Par(v)}_N \cup \bigcup_{v\in S_1 \bs \Desc(\Par(S_2)) } \T^v_N$ gives 
\[
\Phi_s(B)\ge \Phi_s(\ol B) \ge \rc 3 \fc{\pi(S)}{\pi(A)-s}.
\]

\paragraph{Case 2.} $\phi \in B$. Consider $B^c$. If $\pi(B)\le \rc 2$, then $\pi(B^c)\ge \rc 2 \ge s$ and 
\[
\Phi_s(B) = \fc{Q(B, B^c)}{\pi(B)-s} 
\ge \fc{Q(B^c, B)}{\pi(B^c)-s} = \Phi_s(B^c).
\]
Then $\Phi_s(B^c)$ can be bounded as in case 1, noting that we only used the fact that the measure of the set is $\ge s$.
\end{proof}

For the Markov chain in \Cref{a:mc-crem}, we show we can restrict to considering disjoint unions of trees at depth at least $m_0$. Define the equalized distribution $\ol \pi$ by 
\[
\ol \pi (v) = \fc{\mun{|v|}(v)}{N+1}. 
\]
\begin{cor}\label{c:trees-m0-suffice}
Let $R = \fc{\max_{m_0\le m\le n} \Zhatn{m}}{\min_{m_0\le m\le n} \Zhatn{m}}$. 
Then for the Markov chain in \Cref{a:mc-crem},
\[
\Phi_s \ge 
\min_{\scriptsize \begin{array}{c}A:\ol\pi (A)\ge s/R\\A=\Desc^0(S)\\ S\subeq \set{v}{|v|\ge m_0}\end{array}}
\rc{3R(m_0+1)}\cdot  \fc{\ol \pi(S)}{\ol \pi(A) - \fc s{R(m_0+1)}}.
\]
\end{cor}
\begin{proof}
By \Cref{l:trees-suffice}, letting $\pi$ be the stationary distribution,
\[
\Phi_s \ge \minr{A:\pi (A)\ge s}{A=\bigsqcup_{v\in S} \T^v_N} \rc 3\cdot  \fc{\pi(S)}{\pi(A)-s},
\]
so it suffices to lower-bound $\fc{\pi(S)}{\pi(A)-s}$ for $A$ with $\pi(A)\ge s$ and $A=\Desc^0(S)$. Let $S_1 = \set{v\in S}{|v|<m_0}$ and $S_2 = \set{v\in S}{|v|\ge m_0}$. Define $\Desc_{m}(v) = \set{w\in \Desc(v)}{|w|=m}$ and $\Desc_m(S) = \bigcup_{v\in S} \Desc_m (v)$. Let $S'=\Desc_{m_0}(S_1)\cup S_2$, $A'=\Desc^0(S')$.
    Note that for $v\in S_1$, $\pi(v)= \sum_{w\in \Desc_{m_0}(v)} \pi(w)$ by construction, so 
    \begin{align*}
        \fc{\pi(S)}{\pi(A)-s}
        &= \fc{\sum_{v\in S_1} \pi(v) + \sum_{v\in S_2} \pi(v)}{\sum_{v\in S_1} \pi(\T^v_N) + \sum_{v\in S_2} \pi(\T^v_N) - s}\\
        &\ge \fc{\sum_{v\in \Desc_{m_0}(S_1)} \pi(v) + \sum_{v\in S_2} \pi(v)}{(m_0+1)\sum_{v\in \Desc_{m_0}(S_1)} \pi(\T^v_N) + \sum_{v\in S_2} \pi(\T^v_N) - s} = \fc{\pi(S')}{(m_0+1)\pi(A')-s}.
    \end{align*}
    Finally, because $\pi(v) \propto  \ol \pi(v) \Zhatn{|v|\vee m_0}$, we have $c\pi \ge \pi \ge \fc{c}{R}\ol \pi$ for some $c\in (1,R)$, so $\pi(S')\ge \fc{\ol\pi(S')}{R}$ and 
    \begin{align*}
    \fc{\pi(S')}{(m_0+1)\pi(A')-s}&\ge 
    \fc{\fc cR \ol \pi(S')}{(m_0+1) c \ol \pi(A') - s}\\
    &\ge
    \rc R\cdot  \fc{\ol \pi(S')}{(m_0+1)\ol \pi(A')-\fc sR}
    \ge \rc{R(m_0+1)} \fc{\ol \pi(S')}{\ol \pi(A') - \fc{s}{R(m_0+1)}}. 
    \end{align*}
\end{proof}

\subsection{Tail bounds for $\Zhatn{N}$}
\label{s:tail}
We first derive bounds on the moments of $\Zhatn{N}$, and use this to derive both bounds on the upper and lower tails of the distribution.
\begin{lem}\label{l:lim-dist}
    For 
    a CREM under \Cref{a:hi-temp}(\ref{i:1}) ($\be<\bmin$), 
    $\wh{Z}_{\beta,N}$ has moments of order $p<\fc{2\ln 2}{\be^2\amx}$ bounded in terms of $\beta, p,A$. For $p\le 2$,  
    \[
\E \ab{\Zhatn{N} - 1}^p 
\le
\fc{2^{4p+1}}{(p-1) (2\ln 2 - p\be^2 \amx)}.
    \]
    In particular, taking $p = \min\bc{\rc 2 \pa{1+\fc{2\ln 2}{\be^2\amx}},2}$, we obtain 
    \[\E \ab{\Zhatn{N} - 1}^p 
\lesssim \rc{g^2} = \rc{\left(\sqrt{\ln 2} - \beta \sqrt{\frac{\amx}{2}}\right)^2}.
\]
\end{lem}
From this it easily follows that the same bound holds for $\E [\Zhatn{N}^p]$ up to constant factors.
We note for future reference that because 
$\be\sqrt{\amx}\le \sqrt{2\ln 2}$, we have that taking $p$ as above,
\begin{align}\label{l:pgap}
\rc 2 \pa{1+\fc{2\ln 2}{\be^2\amx}} - 1 \gtrsim \fc{g}{\be \sqrt{\amx}} \gtrsim g\quad \implies \quad p-1\gtrsim g. 
\end{align}
One can in fact show that larger moments do not exist. Contrast this with the limiting distribution of $\Zhatn{N}$ for the REM \cite{BKL02}, where at high enough temperature, moments of all orders exist and the distribution is log-normal.
\begin{proof}
    We generalize \cite[Theorem 1]{biggins1992uniform} for the branching random walk to the CREM setting. For convenience, for $i\in [N]$, let 
    \[m_i(\beta):=\E \left[\sum_{u\in\{0,1\}}\exp(\beta Y_{u})\right]
    = 2\E\exp(\be Y_0), \]
    where $Y_u \sim \mathcal{N}(0,N\cdot (A(i/N) - A((i-1)/N)))$ are i.i.d. This is the sum of exponential moments of the new children generated in the $i$-th generation by a $(i-1)$-th generation parent. Below, let $\E = \Ea$. We observe the following fact by the independence of $Y_u$ at different generations:
    \[
    \E \Zn{n} = \prod_{i=1}^n m_i(\beta).
    \]
    Letting $\sF_n$ be the $\si$-algebra generated by the first $n$ generations, 
    \begin{align}\label{eq:pmom-cond}
        \E\left[ |\Zhatn{n} - \Zhatn{n-1}|^p|\sF_{n-1}\right] = \E\left[ \bigg|\sum_{v\in \{0,1\}^{n-1}} 
        \frac{\exp(\beta X_v)}{\prod_{i=1}^{n-1}m_i(\beta)}\left(\frac{ X_{1,v}}{m_n(\beta)} - 1\right)\bigg|^p|\sF_{n-1}\right],
    \end{align}
where $X_{1,v}$ is the sum of exponentials of the edges connecting the $(n-1)$-th generation vertices $v$ and their child vertices, i.e. 
\[
X_{1,v}:= \sum_{
u\in \{0,1\}} \exp(\beta Y_{vu}), \quad  \text{where} \ \ Y_{vu} \sim \mathcal{N}\pa{0,a(n)-a(n-1)}.
\] On the other hand, note that by 
H\"older's inequality,
\begin{align}\label{eq:1edge-bd}
\E \bigg|\frac{X_{1,v}}{m_n(\beta)}-1 \bigg|^p \le 
2^{p-1} \cdot  \E\ba{1^p + \pf{X_{1,v}}{m_n(\be)}^p}
\le 
2^{p}\cdot 
\E\ba{\pf{X_{1,v}}{m_n(\be)}^p}
\le 2^{2p-1} \cdot 
\fc{m_n(p\be)}{m_n(\be)^p}.
\end{align}
Applying Lemma 1 in~\cite{biggins1992uniform}, taking expectation of~\eqref{eq:pmom-cond}, we have
\begin{align}
    \E\left[ |\Zhatn{n} - \Zhatn{n-1}|^p \mid \sF_{n-1} \right] 
    &\le 2^p \sum_{v\in \{0,1\}^{n-1}} \E \ba{\ab{\frac{\exp(\beta X_v)}{\prod_{i=1}^{n-1}m_i(\beta)}\left(\frac{ X_{1,v}}{m_n(\beta)} - 1\right)}^p \mid \sF_{n-1} }\\
    &\le 2^p \cdot \left( \prod_{i=1}^{n-1} \frac{m_i(p\beta)}{m_i(\beta)^p}\right) \cdot 2^{2p-1} \cdot \fc{m_n(p\be)}{m_n(\be)^p}. 
\end{align}
Again by Lemma 1 in~\cite{biggins1992uniform} to the martingale differences $\Zhatn{n}-\wh Z_{\be, n-1}$, 
\[
\E |\wh{Z}_{\beta,N} -1 |^p \le 2^{4p-1} \cdot \sum_{n=1}^N\left( \prod_{i=1}^{n} \frac{m_i(p\beta)}{m_i(\beta)^p}\right).
\]
If $\fc{m_i(p\be)}{m_i(\be)^p}$ is uniformly bounded away from 1, then the $p$-th moment of $\wh Z_{N,\be}$ can be uniformly bounded by an infinite geometric series, independent of $N$.
It suffices for all $i\in [N]$ that
\[
\frac{m_i(p\beta)}{m_i(\beta)^p} \le 2^{-(p-1)} \cdot \exp\left(\frac12 (p^2-p) \beta^2 \sup A'\right) 
= \exp\pa{(p-1)\pa{-\ln 2 + \rc 2 p\be^2 \amx}}<1.
\]
Equivalently, 
\[
-\ln 2 + \frac12 p \beta^2 \amx <0 \Longleftrightarrow \beta^2 < \frac{2\ln 2}{p \cdot \amx}.
\]
Using this, letting $\mathcal{C}:= \exp\pa{- (p-1)\pa{\ln 2 - \rc 2 p\beta^2\amx}}$, we have 
\begin{align*}
\E |\wh{Z}_{\beta,N} -1 |^p & \le 2^{4p-1} \cdot \sum_{n=1}^{N}\mathcal{C}^n\\
&\le  2^{4p-1} \cdot \frac{1}{1-\mathcal{C}} \\
&\le 2^{4p+1} \cdot \frac{1}{(p-1) (2\ln 2 - p\be^2 \amx)},
\end{align*}
where the last inequality follows because for $p\le 2$, $(p-1)\pa{-\ln 2 + \rc 2 p\be^2 \amx}\le \ln 2$, and for $0\le x\le \ln 2$, $e^{-x}\le 1-\rc 2 x$, $\rc{1-e^{-x}} \le \fc{2}{x}$.
\end{proof}

The following lemma shows how to obtain a bound for the tail of the expectation from a $p$th moment bound for some $p>1$. As we have a $p$th moment bound for $\Zn{N}$ for some $p>1$ by \Cref{l:lim-dist}, we can apply this lemma to bound $\E \Zhatn{N} \one_{\Zhatn{N}\ge C\Zoln{N}}$ where $\Zoln{N} = \E \Zhatn{N}$.
\begin{lem}\label{l:Etail-from-moment}
    Let $X$ be a non-negative random variable such that $\E[X^p] \le C (\E X)^p$ for some $p>1$ and let $L\ge 1$. Then 
    \[
\fc{\E X \one_{X\ge L \cdot \E X}}{\E X}
\le C\pa{\rc{L^p} + \rc{L^{p-1}}}\le \fc{2C}{L^{p-1}}.
    \]
\end{lem}
\begin{proof}
We have
\begin{align*}
    \frac{\E X \one_{X\ge L \cdot \E X}}{\E X}
    &\le 
    \Pj(X/\E X\ge L) + 
    \int_L^\iy \Pj(X / \E X \ge t)\,dt\\
    &\le 
    \Pj(X /\E X\ge L) + \int_L^\iy \fc{\E[(X / \E X)^p]}{t^p}\,dt\\
    &\le 
    \Pj(X /\E X\ge L) + \int_L^\iy \fc{C}{t^p}\,dt\\
    &\le \fc{C}{L^p} + \fc{C}{L^{p-1}} \le \fc{2C}{L^{p-1}}.
\end{align*}
\end{proof}

We also show that $\Zhatn{N}$ is not too small with reasonable probability, by using \Cref{l:lim-dist} in conjunction with the Paley-Zygmund inequality.
\begin{lem}\label{l:Z-big}
Consider 
    a CREM under \Cref{a:hi-temp}(\ref{i:1}) ($\be<\bmin$). 
For any constant $\ep_1>0$, we have 
    \[
    \Pj(\Zhatn{N} > \ep_1)
    \ge g^{O(1/g)}
    \]
where the $O(\cdot)$ hides a constant depending on $\ep_1$.
\end{lem}
\begin{proof}
    The $L^p$ version of the Paley-Zygmund inequality says that for $p>1$ and a positive random variable $Z$,
    \[
\Pj(Z>\te \cdot \E Z) \ge \fc{(1-\te)^{\fc{p}{p-1}} (\E Z)^{\fc{p}{p-1}}}{\E[Z^p]^{\rc{p-1}}}.
    \]
    Applying this to $\Zhatn{N}$ and using \Cref{l:lim-dist}, we obtain for some constants $C_1,C_2>0$ that 
    \[
\Pj(\Zhatn{N} > \te) 
\ge \fc{(1-\te)^{\fc{p}{p-1}}}{1/(C_1g^2)^{\rc{p-1}}}
\ge (C_1(1-\te)g^2)^{\fc{C_2}{g}}.
    \]
\end{proof}


\subsection{Bounding $s$-conductance for the CREM}
\label{s:conductance-crem}
The following bounds the lower tail of a non-negative random variable, under weaker conditions than a typical concentration inequality.
\begin{lem}\label{l:prob-small}
    If $X_1,\ldots, X_k$ are i.i.d. copies of a non-negative random variable $X$ with $\Pj(X\ge \ep_1)\ge \ep_2$ and $\sumo jk p_j=1$, $p_j\ge 0$, then 
    \[
\Pj\pa{\sumo jk p_jX_j \le \fc{\ep_1\ep_2}2} \le \exp\pa{-\fc{\ep_2^2}{2\max_j p_j}}.
    \]
\end{lem}
\begin{proof}
Note that if $\sumo jk p_jX_j \le \fc{\ep_1\ep_2}2$, then because $\sumo jk p_j \ep_1 \one_{X_j\ge\ep_1} \le 
\sumo jk p_j X_j\le \fc{\ep_1\ep_2}2$, we have $\sumo jk p_j \one_{X_j\ge \ep_1}\le  \fc{\ep_2}2$.
Note that by assumption $\E \left(\sum_{j=1}^k p_j\one_{X_j\ge \ep_1}\right) \ge \ep_2$. 
    Hence by the Chernoff bound,
    \begin{align*}
    \Pj\pa{\sumo jk p_jX_j \le \fc{\ep_1\ep_2}2}\le
        \Pj\pa{\sumo jk p_j \one_{X_j\ge \ep_1}\le  \fc{\ep_2}2}
        &\le \exp\pa{-\fc{2(\ep_2/2)^2}{\sumo jk p_j^2}}
        \le \exp\pa{-\fc{2(\ep_2/2)^2}{\max_j p_j}}.
    \end{align*}
\end{proof}
Later, we will apply this to $\Zhatn{N}$ combining with Lemma~\ref{l:Z-big}. In the following part, let 
\[
h_{\ge L}(x) := x\one_{x\ge L}.
\]
\begin{lem}\label{l:big-contrib}
Suppose that $X$ is a positive random variable with $\E X=1$, $\E X^p \le C$, and 
$\Pj\pa{X\ge \ep_1}\ge \ep_2$.
Let $X_1,\ldots, X_k$ be i.i.d. copies of $X$.
\begin{enumerate}
    \item For $L\ge 1$, 
\[
\Pj\pa{\fc{\sumo jk p_j h_{\ge L}(X_j)}{\sumo jk p_j X_j} \ge \ep}
\le \exp\pa{-\fc{\ep_2^2}{2\max_j p_j}} + \fc{4C}{\ep \ep_1\ep_2 L^{p-1}}.
\]
Hence, we have
\[
\Pj\pa{\fc{\sumo jk p_j h_{\ge L}(X_j)}{\sumo jk p_j X_j} \ge \ep}\le \de
\]
if $\max_j p_j \le \fc{\ep_2^2}{2\ln \pf{2}{\de}}$ and $L\ge \pf{8C}{\ep\ep_1\ep_2\de}^{\rc{p-1}}$.
\item We have 
\[
\E \ba{\fc{\sumo jk p_j h_{\ge L}(X_j)}{\sumo jk p_j X_j}}\le \exp\pa{-\fc{\ep_2^2}{2\max_j p_j}} + \fc{4C}{\ep_1\ep_2 L^{p-1}}.\]
\end{enumerate}
\end{lem}
\begin{proof}
    For the first part, if the event holds, then either 
    \[
\sumo jk p_j X_j  \le \fc{ \ep_1\ep_2}{2}
\qquad \text{or} \qquad 
\sumo jk p_j X_j \one_{X_j\ge L} \ge \fc{\ep\ep_1\ep_2}2
.
    \]
    The first event has probability at most $\exp\pa{-\fc{\ep_2^2}{2\max_j p_j}}$ by \Cref{l:prob-small}. For the second event, by \Cref{l:Etail-from-moment},
    \begin{align*}
        \E \ba{\sumo jk p_j X_j \one_{X_j\ge L}}
        = \E \ba{X\one_{X\ge L}} \le 
        \fc{2C}{L^{p-1}}
    \end{align*}
    so by Markov's inequality,
    \begin{align*}
        \Pj\pa{\sumo jk p_j X_j \one_{X_j\ge L} \ge \fc{\ep\ep_1\ep_2}2}
        \le \fc{4C}{\ep\ep_1\ep_2L^{p-1}}.
    \end{align*}
    For the second part, by \Cref{l:Etail-from-moment},
\begin{align*}
    \E \ba{\fc{\sumo jk p_j h_{\ge L}(X_j)}{\sumo jk p_j X_j}}
    &\le 
    \Pj\pa{\sumo jk p_j X_j < \fc{\ep_1\ep_2}2} + \E \ba{\one_{\sumo jk p_j X_j \ge \fc{\ep_1\ep_2}2}\fc{\sumo jk p_j h_{\ge L}(X_j)}{\sumo jk p_j X_j}}\\
    &\le \exp\pa{-\fc{\ep_2^2}{2\max_j p_j}} + \fc{2}{\ep_1\ep_2}\cdot \E \ba{\sumo jk p_j h_{\ge L}(X_j)} \\
    &\le \exp\pa{-\fc{\ep_2^2}{2\max_j p_j}} + \fc{4C}{\ep_1\ep_2 L^{p-1}}.
\end{align*}
\end{proof}

\begin{proof}[Proof of Theorem~\ref{t:main-hi-temp} using Markov chain]
Let $\Pj = \Pa$. 
By \Cref{l:ratio-1}, for \\
$m_0 =\ce{\fc{C}{g^2}\pa{\fc{\amx}{g^2}\ln \pf{\amx}{g\de} + \ln \prc{\ep_2}}}$
for appropriate $C$, we have that 
\[
\Pj\pa{
\max_{m\in [m_0,N]}\ab{\fc{\Zhatn{m_0}}{\Zhatn{m}}-1}\ge \rc3 \text{ or }
\max_{|v|\ge m_0} \pv{v} \ge 
\fc{\ep_2^2}{N^2\ln \pf{4N^2}{\de}}
}\le \fc{\de}2.
\]
Let $\ptv{v}=\pv{v}$ if this event holds and $\ptv{v}=\rc{2^{|v|}}$ otherwise. 
Let $p'=\min\bc{\rc 2 \pa{1+\fc{2\ln 2}{\be^2\amx}},2}$. 
From \eqref{l:pgap}, we have $p'-1\gtrsim \rc{g}$.
By Lemma~\ref{l:lim-dist} (in conjunction with \Cref{l:renorm}),  $C(\be):= \max_{v\in \T_N}\E[(\Zhatnv{N}{v})^{p'}] = O\prc{{g}^2}$.
By \Cref{l:Z-big} (in conjunction with \Cref{l:renorm}), we have $\Pj\pa{\Zhatnv{N}{v}\ge \ep_1}\ge \ep_2$ for $\ep_1=\rc 2$ and some $\ep_2= {g}^{O(1/g)}$. 
By \Cref{l:big-contrib}, if 
$L\ge \pf{16C(\be) N^4}{\ep'\ep_1\ep_2\de}^{\rc{p'-1}}$, then 
for each $m_0\le m\le m+n\le N$, by first conditioning on $\mathscr F_m = \si(\set{X_v}{|v|\le m})$,
\begin{align*}
    \Pj\pa{
\fc{\sum_{|v|=m} \ptv{v} h_{\ge L}(\Zhatnv{m+n}{v})}{\sum_{|v|=m} \ptv{v} \Zhatnv{m+n}{v}}\ge \fc{\ep'}{N^2}
    }\le \fc{\de}{2N^2},
\end{align*}
where $\Zhatnv{m+n}{v} = \sum_{|w|=n}X_{vw}$. 
Then 
\begin{align*}
 \Pj\pa{
\exists m\ge m_0,\exists n,\,
\fc{\sum_{|v|=m}\ptv{v} h_{\ge L}(\Zhatnv{m+n}{v})}{\sum_{|v|=m} \ptv{v} \Zhatnv{m+n}{v}}\ge \fc{\ep'}{N^2}
}\le \fc{\de}2.
\end{align*}
Let $R = \fc{\max_{m_0\le m\le n} \Zhatn{m}}{\min_{m_0\le m\le n} \Zhatn{m}}$. 
A coupling argument and union bound shows
\begin{align}\label{e:event}
 \Pj\pa{
R> 2 \text{ or }
\exists m\ge m_0,\exists n,\,
\fc{\sum_{|v|=m} \pv{v} h_{\ge L}(\Zhatnv{m+n}{v})}{\sum_{|v|=m} \pv{v} \Zhatnv{m+n}{v}}\ge \fc{\ep'}{N^2}
}\le \de.
\end{align}
Let $G$ be the complement of the above event, and suppose $G$ holds. 
By \Cref{c:trees-m0-suffice}, it remains to bound the conductance of all subtrees at depth at least $m_0$ in the equalized measure.
For $s=\ep' R$, take $A, S$ satisfying the constraints in \Cref{c:trees-m0-suffice} with $S= \set{v\in A}{\Par(v)\nin A}$. We need to lower-bound $\fc{\ol \pi(S)}{\ol \pi(A) - \fc{s}{R(m_0+1)}}$; it suffices to lower-bound $\fc{\ol\pi(S)}{\ol \pi(A)}$. We have, by partitioning the vertices in $A$ based on the depth of their ancestor in $S$ and the depth from that ancestor,
\[
\ol \pi(A) =\rc{N+1} \sum_{m=m_0}^N \sum_{n=0}^{N-m} 
\pa{\fc{\sum_{v\in S,\,|v|=m,\,\Zhatnv{m+n}{v}\ge L} \pv{v} \Zhatnv{m+n}{v} + \sum_{v\in S,\,|v|=m,\,\Zhatnv{m+n}{v}< L} \pv{v} \Zhatnv{m+n}{v} }{\sum_{|w|=m}\pv{w} \Zhatnv{m+n}{w}}}
\]
Because $G$ holds, we can bound the first sum by 
\begin{align*}
    \sum_{m=m_0}^N \sum_{n=0}^{N-m} 
\pa{\fc{\sum_{v\in S,\,|v|=m,\,\Zhatnv{m+n}{v}\ge L} \pv{v} \Zhatnv{m+n}{v} }{\sum_{|w|=m}\pv{w} \Zhatnv{m+n}{w}}}
&\le \fc{N^2}2 \cdot \fc{\ep'}{N^2} \le \fc{\ep'}2.
\end{align*}
Because $\ol \pi(A)\ge \ep'$, we have
\begin{align*}
    \ol \pi(A) \le \fc{2}{N+1} \sum_{m=m_0}^N \sum_{n=0}^{N-m} 
\pa{\fc{ \sum_{v\in S,\,|v|=m,\,\Zhatnv{m+n}{v}< L} \pv{v} \Zhatnv{m+n}{v} }{\sum_{|w|=m}\pv{w} \Zhatnv{m+n}{w}}}.
\end{align*}
Now, $\ol \pi(S)= \rc{N+1} \sum_{m=m_0}^N 
\pa{\fc{ \sum_{v\in S,\,|v|=m} \pv{v} }{\sum_{|w|=m}\pv{w} }}$, so
\begin{align*}
    \fc{\ol \pi(S)}{\ol \pi(A)}&
    \ge \fc
    {\sum_{m=m_0}^N 
\pa{\fc{ \sum_{v\in S,\,|v|=m} \pv{v} }{\sum_{|w|=m}\pv{w} }}}
    {2 \sum_{m=m_0}^N \sum_{n=0}^{N-m} 
\pa{\fc{ \sum_{v\in S,\,|v|=m,\,\Zhatnv{m+n}{v}< L} \pv{v} \Zhatnv{m+n}{v} }{\sum_{|w|=m}\pv{w} \Zhatnv{m+n}{w}}}}.
\end{align*}
We have for each $m\in [m_0,N]$ 
that the ratio of the summands above is
\begin{align*}
&\fc{ \sum_{v\in S,\,|v|=m} \pv{v} }{\sum_{|w|=m}\pv{w} }
    \Bigg/ 
    \sum_{n=0}^{N-m} 
\pa{\fc{ \sum_{v\in S,\,|v|=m,\,\Zhatnv{m+n}{v}< L} \pv{v} \Zhatnv{m+n}{v} }{\sum_{|w|=m}\pv{w} \Zhatnv{m+n}{w}}}\\
&\ge 
\fc{ \sum_{v\in S,\,|v|=m} \pv{v} }{\sum_{|w|=m}\pv{w} }
    \Bigg/ 
    L \cdot \sum_{n=0}^{N-m} 
\pa{\fc{ \sum_{v\in S,\,|v|=m} \pv{v} }{\sum_{|w|=m}\pv{w} \Zhatnv{m+n}{w}}}\\
&\ge
\rc{L} 
\pa{
\sumz{n}{N-m} \fc{\sum_{|w|=m}\pv{w}}{{\sum_{|w|=m}\pv{w} \Zhatnv{m+n}{w}}}}^{-1}
\ge \rc{RLN}.
\end{align*}
Hence $\fc{\ol\pi(S)}{\ol\pi(A)}\ge \rc{2RLN}$ and when $R\le 2$,
\[
\rc{3R (m_0+1)}\cdot  \fc{\ol \pi(S)}{\ol \pi (A)}\ge \rc{3R N} \cdot \rc{2RNL} \ge \rc{24N^2L}.
\]
By \Cref{c:trees-m0-suffice}, this shows 
$\Phi_{2\ep'} \ge \Phi_s \ge \rc{24N^2L}$.
Note that we start the chain at $V_0=\phi$, which is a $2N$-warm start. 
We note that for $N\ge 16$, when $R\le 2$,  all levels have weight $<  \fc 2N \le \rc 8$ and hence all atoms have size $<\rc 8$. (We can simply ensure $m_0\ge 16$ to take care of the case where $N<16$.) 
Then by \Cref{l:conductance-mixing} and choosing $L= \pf{16C(\be) N^4}{\ep'\ep_1\ep_2\de}^{\rc{p'-1}}$, $\ep_1=\rc 2 $, $\ep_2 = {g}^{O(1/g)}$, 
\begin{align*}
\TV(\mu_0T^k, \pi) 
&\le 
2N \pa{2\ep' + \rc 2 e^{-k\Phi_{2\ep}^2/2}}\\
&\le 
2N \pa{2\ep' + \rc 2 e^{-\Om\pf k{(NL)^2}}}
\le 
2N \pa{2\ep' + e^{-k\pf{{g}^{1/g} \de}{N}^{O\prc {g}}}}.
\end{align*}
Now choose $\ep'=\fc{\ep}{16N^2}$ and $k=\pf{N}{{g}^{1/g}\de}^{\fc{C}{g}}\ln \pf N\ep$ for an appropriate constant $C$ to bound this by $\fc{\ep}{2N}$. Then restricting to level $N$ (which has at least $\fc{1}{2N}$ of the mass under $\mun{N}$), 
we have 
$\TV(\mu_0 T^k |_{\{0,1\}^N}, \mun{N}) \le \ep$.

Finally, we compute the total running time. It takes $O(2^{m_0}) = \pf{\amx}{\de g}^{O(\amx/g^4)}$ time and queries for the preprocessing step of defining the distribution $\pi$. The Markov chain takes $\pf{N}{\ep {g}^{1/g}}^{O(1/g)}$ steps to run. Finally, the number of runs necessary to obtain a sample at the $N$th level is a geometric random variable with success probability at least $\rc{2N}$, so with probability $\ge 1-\de$, the number of runs is bounded by $O\pa{N\ln \prc\de}$. Replacing $\de$ by $\de/2$ takes care of both failure in the the CREM randomness and in the running time.
\end{proof}

\subsection{Upper bound on spectral gap}
\label{s:spectral}
We prove \Cref{t:ub-gap}: it is not the case that CREM has a good spectral gap. Towards this, we first show \Cref{l:power-law}: the tails of $\Zhatn{N}$ are lower-bounded by the tails of a power law distribution.
This allows us to prove \Cref{t:ub-gap} by nothing that at some depth $cN$, there are exponentially many vertices, so with good probability one of the vertices $v$ has an exponentially large partition function. Provided that the probability mass of $v$ at level $cN$ is not extraordinarily large, the subtree $\T^v_N$ will likely have exponentially small conductance.

\begin{lem}[Power law tail lower bound]\label{l:power-law}
Consider the CREM with $A(x)=x$. 
For any $\be>0$, there exist constants $c,C,\al>0$ depending on $\be$ such that for 
$t = \exp(c N)$, we have 
\[
\Pa (\Zhatn{N}\ge t) \ge \frac{C}{t^\al}.
\]
\end{lem}

\begin{proof}[Proof of Lemma~\ref{l:power-law}]

For each $n$, consider the following event 
\[
A_n := \{\exists v,\, |v|=n,\, \pv{v} \ge c_1,\, \be X_v \ge C_2 n\},
\]
for constants $c_1>0, C_2>\fc{\ln \Ea Z_{n,\be}}n = \ln 2 + \fc{\be^2}2$ to be chosen. We claim that 
\begin{align}\label{e:plt-ind-claim}
\Pa\pa{\bigcap_{m=1}^n A_m } \ge e^{-c_3 n}, \ \text{for some $c_3>0$}.
\end{align}
Now suppose this claim holds, and let $v^*$ be such that the condition in $A_n$ is satisfied for $v^*$. Note that on $\bigcap_{m=1}^n A_m$, we have 
\begin{align*}
    \Zn{n}  = \sum_{|v| = n} e^{\beta X_v} \ge e^{\beta X_{v^*}} \ge e^{C_2 n}.
\end{align*}
Thus 
\[
  \Pa\left(\Zn{n} \ge e^{C_2 n}\right) \ge e^{-c_3n}.
\]
On the other hand, observe that
\[
\Zhatn{n} \ge \exp\pa{\pa{C_2 - \fc{\ln\Ea \Zn{n}}{n}}n} \quad \Longleftrightarrow \quad \Zn{n} \ge e^{C_2 n}.
\]
Thus 
\[
\Pa\left(\Zhatn{n} \ge e^{(C_2 - \ln\E \Zn{n})n}\right) \ge e^{-c_3n}
\]
which implies the desired conclusion.

It remains to prove \eqref{e:plt-ind-claim}. 
We proceed inductively; it suffices to prove 
\[
\Pa\pa{A_{n+1} \Big| \bigcap_{m=1}^{n}A_m} \ge e^{-c_3}.
\]
For this, it suffices to show that given the existence of $v^*$ satisfying the condition in $A_n$, with probability lower-bounded by a nonzero constant, one can find a child vertex of $v^*$ which satisfies the condition in $A_{n+1}$. Namely, with positive probability, either 
\[
p_{v^*0} \ge c_1 \quad \text{and} \quad \beta X_{v^* 0}\ge C_2(n+1),
\]
or 
\[
p_{v^*1} \ge c_1 \quad \text{and} \quad \beta X_{v^* 1}\ge C_2(n+1).
\]
First, by definition, $\beta X_{v^*0} = \beta( X_{v^*} + Y_{v^*0})$. 
Let 
\begin{align}
    \label{e:e1}
    E_1 = \{\be \max\{Y_{v^*0}, Y_{v^*1}\}>C_2\}.
\end{align}
This event holds with positive probability depending only on $C_2$ and $\be$, and under $A_n$ and this event, either $\be X_{v^*0}>C_2(n+1)$ or $\be X_{v^*1}>C_2(n+1)$. 

Now write 
\begin{align*}
\Zn{n+1} &= \Zn{n} \sum_{|v|=n} \pv{v} \cdot (e^{\be Y_{v0}}+e^{\be Y_{v1}})\\
&= \Zn{n} \ba{
p_{v^*}\cdot (e^{\be Y_{v^*0}}+e^{\be Y_{v^*1}})
+ \sum_{v\ne v^*} \pv{v} (e^{\be Y_{v0}}+e^{\be Y_{v1}})
}.
\end{align*}
Note 
\[
\Ea \ba{\sum_{v\ne v^*} \pv{v} (e^{\be Y_{v0}} + e^{\be Y_{v1}})} = 2(1-\pv{v^*}) e^{\be^2/2}
\]
so by Markov's inequality,
\begin{align*}
    \Pa \pa{
     E_2
    } &\ge 1 - \fc{2e^{\be^2/2}}{c_4}\\
    \text{where }
    E_2 &=\bc{\sum_{v\ne v^*} \pv{v} (e^{\be Y_{v0}} + e^{\be Y_{v1}}) \le (1-\pv{v^*}) c_4 }.
\end{align*}
Choosing $c_4 > 2e^{\be^2/2}$, this probability is bounded below by a positive constant.

Then under the events $A_n$, $E_1$, and $E_2$,
\begin{align*}
    \max\{\pv{v^*0}, \pv{v^*1}\}
    &\ge 
    \fc{\pv{v^*}\max\bc{e^{\be Y_{v^*0}},e^{\be Y_{v^*1}}}}{\pv{v^*}(e^{\be Y_{v^*0}}+e^{\be Y_{v^*1}}) + \sum_{v\ne v^*} \pv{v}(e^{\be Y_{v0}}+e^{\be Y_{v1}}) }\\
    &\ge    
    \fc{\pv{v^*}\max\bc{e^{\be Y_{v^*0}},e^{\be Y_{v^*1}}}}{2\pv{v^*}\max\bc{e^{\be Y_{v^*0}},e^{\be Y_{v^*1}}} + \sum_{v\ne v^*} \pv{v}(e^{\be Y_{v0}}+e^{\be Y_{v1}}) }\\
    &\ge \fc{\pv{v^*}C_2}{2\pv{v^*}C_2 + \sum_{v\ne v^*} \pv{v}(e^{\be Y_{v0}}+e^{\be Y_{v1}}) }\\
    &\ge \fc{\pv{v^*}C_2}{2\pv{v^*}C_2 + (1-\pv{v^*}) c_4} \\
    &\ge \fc{c_1C_2}{2c_1C_2 + (1-c_1) c_4}.
\end{align*}
Now for any fixed $c_1>\rc2$, we can choose $C_2$ large enough depending on $c_4$ so that this is $\ge c_1$, as needed. Then $A_n \cap E_1\cap E_2\subeq A_{n+1}$. 

Finally, we note that $E_1$ and $E_2$ are independent conditional on $\sF_n$, so 
\[
\Pa\pa{A_{n+1} \Big| \bigcap_{m=1}^{n}A_m} \ge 
\Pa\pa{E_1\cap E_2 \Big| \bigcap_{m=1}^{n}A_m} \ge e^{-c_3}
\]
for some $c_3$, completing the induction step.
\end{proof}

\begin{proof}[Proof of \Cref{t:ub-gap}]
    By \Cref{l:power-law}, there exist $c', c''>0$ such that given any $c>0$, for $N$ large enough,
    \[
\Pa \pa{\Zhatn{\fl{cN}} > e^{cc'N}}
\ge e^{-cc''N}. 
    \] 
    Choose $c\le \rc 2$ small enough so that 
    \[
\Pa \pa{\Zhatn{\fl{cN}} > e^{cc'N}}
\ge 2^{-N/4}.
    \]
    Note $(\Zhatnv{N}{v})_{|v|=\ce{(1-c)N}}$ are $2^{\ce{(1-c)N}}$ i.i.d. random variables.  Hence
\begin{align}
\nonumber 
\Pa\pa{\Zhatnv{N}{v} > e^{cc'N}
\text{ for some }|v|=\ce{(1-c)N}}
&\ge 
1-(1-2^{-N/4})^{2^{\ce{(1-c)N}}}\ge
1-(1-e^{-N/4})^{2^{N/2}}\\
&\ge 1 - e^{-2^{-N/4} 2^{N/2}}
= 1-e^{-2^{N/4}}.
\label{e:ub-gap-1}
\end{align}
Suppose such a $v$ exists; choose the first such $v$ (under some fixed order). Consider the conductance of $\T^v_N$, the subtree rooted at $v$. By symmetry and Markov's inequality,
\begin{align}
\nonumber
    &\Ea \ba{\fc{\pv{v}}{\sum_{|w|=\ce{(1-c)N}} \pv{w}}}
    = 
    \rc{2^{\ce{(1-c)N}}} \sum_{|v|=\ce{(1-c)N}} \Ea \ba{\fc{\pv{v}}{\sum_{|w|=\ce{(1-c)N}} \pv{w}}} = \rc{2^{\ce{(1-c)N}}} \le \rc{2^{N/2}}\\
    \label{e:ub-gap-2}
    &\implies
    \Pa\pa{\fc{\pv{v}}{\sum_{|w|=\ce{(1-c)N}} \pv{w}}<\rc{2^{N/4}}} \ge 1-
    2^{N/4} \cdot  \Ea \ba{\fc{\pv{v}}{\sum_{|w|=\ce{(1-c)N}} \pv{w}}} = 
    1 - 2^{-N/4}.
\end{align}
Under the events in~\eqref{e:ub-gap-1} and~\eqref{e:ub-gap-2},
\begin{align*}
    \Phi &\le 
\fc{Q(\Desc^0(v), \Desc^0(v)^c)}{\min\{\pi(\Desc^0(v)),\pi( \Desc^0(v)^c)\}}\\
&\le \max\bc{
\fc{\pv{v}}{\pv{v}\pa{1 + \Zhatnv{N}{v}}} , 
\fc{\pv{v}}{\sum_{|w|=\ce{(1-c)N}}\pv{w}}
}\\
&\le \max\bc{e^{-cc'N}, 2^{-N/4}}.
\end{align*}
The result now follows from 
that the spectral gap $\ga$ satisfies $\ga \le 2\Phi$ \cite{lawler1988bounds,sinclair1989approximate}. 
\end{proof}

\section{Sequential sampler}
\label{s:seq}

The main challenge in turning \Cref{t:main-Z} into a guarantee for the sequential sampler in our main \Cref{t:main-hi-temp} is that the distribution after fixing the initial $t$ coordinates is no longer a CREM, but a tilted one. In order to utilize our result under the CREM (\Cref{t:main-Z}), we need to show a change-of-measure result between the tilted and original CREM (\Cref{l:com} in \Cref{s:com}). This is essentially a statement about contiguity as $N\to \iy$. We note that a contiguity argument was also an important part of the result \cite{EMS22} for sampling from the SK model.

For ease of presentation, we first show an infinite, non-quantitative version (\Cref{t:contiguous} in \Cref{s:contiguity}), then derive a quantitative bound for finite $n$ (\Cref{l:com-finite} in \Cref{s:com-finite}). 


We will state definitions to work for both the finite and infinite setting. 
We first define the infinite continuous random energy model as follows. 
\begin{df}\label{d:crem-iy}
    Let $a:\R_{\ge 0}\to \R_{\ge 0}$ be a non-decreasing function. Let $\T = \mathbb T_\iy:=\bigcup_{n=0}^{\iy}\{0,1\}^n$ denote the (vertices of the) infinite binary tree. 
    Define the probability measure of the infinite CREM $\PCREM = \Pan{\iy}$ as follows. The probability space is $\Om:=\R^{\T}$ with the Gaussian product measure where
    \[
\om_u \sim \begin{cases}
    \cal N(0, a(0)), & u=\phi \\
    \cal N(0, a(|u|)-a(|u|-1)), & u \in \T\bs \{\phi\}. 
\end{cases}
    \]
    Define the random variables $Y_u(\om) = \om_u$. Define a filtration with $\sF_n := \si(\set{Y_u}{|u|\le n})$, the $\si$-algebra generated by the random variables up to depth $n$.
    For $u\in \T_{\iy}$, let $X_{v_1,\ldots, v_n} = \sumo mn Y_{v_1,\ldots, v_m}$.
Define $\Znv{n+m}{v}$ and $\Zhatnv{n+m}{v}$ as in~\eqref{e:Znv} and~\eqref{e:Zhatnv}. 
Define a measure on $\{0,1\}^\N$ (for cylinder sets) as follows: for any $|v|=n$,\footnote{For the CREM, all the $\Zhatv{w}$ are well-defined almost surely: the sequence of random variables $\Zhatnv{n+m}{v}$ converges a.s. because they form a martingale that is uniformly bounded in $L^p$ for some $p>1$ by \Cref{l:lim-dist}.}
\begin{align*}
\mu_\be(C_v)
&= \fc{e^{\be X_{v}} \Zhatv{v}}{\sum_{|w|=n} e^{\be X_w} \Zhatv{w}},&
\text{where } 
C_v &= \set{vw}{w\in \{0,1\}^\N}\subeq \{0,1\}^\N\\
&&\Zhatv{a,w} &= \lim_{m\to \iy} \Zhatnv{n+m}{a,w}.
\end{align*}
It is straightforward to check this is consistent and hence defines a measure on $\{0,1\}^\N$ by Kolmogorov's extension theorem.

We also define 
\[\ZhatanN{n}{N} := \fc{\Zn{N-n}}{2^{N-n} e^{\fc{\be^2}2(a(N)-a(n))}}\quad \text{ and }\quad 
\ZhatanN{n}{\iy} = \lim_{m\to \iy} \ZhatanN{n}{n+m}.\] 
\end{df}
We define several more measures associated with a CREM (finite or infinite) as follows. We allow $N=\iy$ in the below definition (note $N-n=\iy$).
\begin{df} \label{d:dists}
Define $\muNn{n}:\Om_N \to \cal P(\{0,1\}^n)$ to be the marginal distribution of the first $n$ coordinates of $\mun{N}$, i.e., 
\[
\muNn{n}(v) = 
\mun{N}\pa{\set{vw}{w\in \{0,1\}^{N-n}}}, \quad v\in \{0,1\}^n.
\]
Similarly define $\mu_{\be}|_n$ to be the marginal distribution of the first $n$ coordinates of $\mu_\be$. 

Define the measure rooted at $v$ ($|v|=n$) as follows: $\munv{N}{v}$ is a random measure on $\Om_{N-n}$ such that for $A\subeq \Om_{N-n}$,
\[
\munv{N}{v} (A)
= \fc{\mun{N} \pa{\set{vw}{w\in A}}}{\mun{N} \pa{\set{vw}{w\in \Om_{N-n}}}}. 
\]
Let $\muv{v} = \munv{\iy}{v}$. 

Let $\nun{n}: \Om_N \to \cal P(\{0,1\}^{N-n})$ be the random measure $\munv{N}{v}$ averaged over $v$, where $v\sim \muNn{n}$. 

Define $\Qatr$ as a probability measure on $\Om_{N-n}$ follows: given $\om\in \Om_{N}$, pick $v\sim \muNn{n}$ 
and let $\om'=\bfT_v(\om)$, where $\bfT_v:\Om_{N} \to \Om_{N-n}$ is the translation map defined by 
\begin{align}\label{e:translation}
(\mathbf T_v(\om))_w = \om_{vw}.
\end{align}
Then $\Qatr$ is the distribution of $\om'$. Let $\PCREMn = \Qatri$.

For a function $a:\R_{\ge0}\to \R$, define $\bfT_n a(m) = a(n+m)$, and define $\Patr:=\Pj^{\bfT_n a}_{N-n}$.
\end{df}
It is straightforward to see that $\mu_{\be,N-n}$ on $\Qatr$ has the same distribution as $\nun{n}$ on $\Pan{N}$.
They both describe the distribution when we choose $v$ from the marginal distribution at depth $n$, and then consider the Gibbs measure associated with the subtree with depth $N-n$ starting at that $v$, where the choice of $v$ is averaged out. For $\mu_{\be,N-n}$ on $\Qatr$, the choice of $v$ is done in $\Qatr$, which for $\nun{n}$ on $\Pan{N}$, the choice of $v$ is done in $\nun{n}$.

\subsection{Change of measure}
\label{s:com}
We give a change-of-measure result between $\Qatr$, where we select a vertex according to the CREM ($v\sim \muNn{n}$) and consider the subtree, and $\Patr=\Pj^{\bfT_n a}_{N-n}$, which is equivalent to selecting a vertex independent of the CREM and considering the subtree.
\begin{lem}\label{l:com}
Consider a CREM with unnormalized covariance function $a:[0,N]\to \R_{\ge 0}$. (We allow $N=\iy$.) Then 
    $\dd{\Qatr}{\Patr} = f_n(\ZhatnN{n}{N})$, where 
    \[
f_n(z) = 
    \E^a_N\ba{\fc{2^n \pv{v_0} z}{\pv{v_0}z + \sum_{w\in \{0,1\}^n\bs\{v_0\}} \pv{w} \Zhatnv{N}{w}}}
    \text{ for any fixed }v_0\in \{0,1\}^n,
    \]
    and $f_n$ 
    is an increasing function.
\end{lem}
\begin{proof}
Note that we have the following translation invariance property: for fixed $v$, $\mathbf T_v(\om)$ under $\Pa$ has the same distribution as $\om$ under $\Patr$. 
Because $\ZhatnN{n}{N}(\bfT_v(\om)) = \Zhatnv{N}{v}(\om)$, we get that 
\[
\Pa (\mathbf T_v(\om)\in A \mid \Zhatnv{N}{v} = z)
= 
\Patr (\om \in A \mid \ZhatnN{n}{N} = z). 
\]
Below, we denote $\E = \Ea$. Fix $v_0\in \{0,1\}^n$. We have
    \begin{align*}
    \Patr(A) 
    &= \E\ba{\Patr(A\mid \ZhatnN{n}{N})}.
\end{align*}
By definition of $\Qatr$ and symmetry, for fixed $v_0$ with $|v_0|=n$, we have 
\begin{align*}
    \Qatr  (A)&= 
    \sum_{|v|=n}\E\ba{\E \ba{ \fc{\pv{v}\Zhatnv{N}{v}}{\sum_{|w|=n} \pv{w} \Zhatnv{N}{w}} \Pa(\mathbf T_v(\om)\in A\mid \Zhatnv{N}{v}) \mid (\Zhatnv{N}{v'})_{|v'|=n}, (\pv{v'})_{|v'|=n}}} \\
    &= \E\ba{\E\ba{\fc{2^n \pv{v_0}\Zhatnv{N}{v_0}}{\sum_{|w|=n}\pv{w} \Zhatnv{N}{w}} \Patr(A \mid \ZhatnN{n}{N} = \Zhatnv{N}{v_0}) \mid (\Zhatnv{N}{v})_{|v|=n}, (\pv{v})_{|v|=n}}}\\
    &= \E\ba{\E\ba{\fc{2^n \pv{v_0}\Zhatnv{N}{v_0}}{\sum_{|w|=n}\pv{w} \Zhatnv{N}{w}}\E\ba{ \Patr(A \mid \ZhatnN{n}{N} = \Zhatnv{N}{v_0}) \mid (\Zhatnv{N}{v})_{|v|=n}, (\pv{v})_{|v|=n}} \mid \Zhatnv{N}{v_0}}}\\
    &= \E\ba{\E\ba{\fc{2^n \pv{v_0}\Zhatnv{N}{v_0}}{\sum_{|w|=n}\pv{w} \Zhatnv{N}{w}} \Patr(A \mid \ZhatnN{n}{N} = \Zhatnv{N}{v_0}) \mid \Zhatnv{N}{v_0}}}\\
    &= \E\ba{\E\ba{\fc{2^n \pv{v_0}\Zhatnv{N}{v_0}}{\sum_{|w|=n}\pv{w} \Zhatnv{N}{w}}  \mid \Zhatnv{N}{v_0}}\Patr(A \mid \ZhatnN{n}{N} = \Zhatnv{N}{v_0})}.
\end{align*}
where:
\begin{itemize}
\item In the first line, we sum over the vertices at depth $n$, of the probability we choose that vertex times the CREM disorder measure at depth $N-n$ for the subtree at that vertex.
\item In the second line, $\Patr(A\mid \ZhatnN{n}{N}=\Zhatnv{N}{v_0})$ is interpreted as the function $\Patr\pa{A\mid \ZhatnN{n}{N}=\cdot}$ with the value of $\Zhatnv{N}{v_0}$ substituted in, i.e., the argument of the outer expectation is interpreted as \[\E_{\om\sim \Pa} [\cdots \E_{\om'\sim \Patr} [\one_{\om'\in A} \mid \ZhatnN{n}{N}(\om') = \Zhatnv{N}{v_0}(\om)]].\] 
    \item In the third line we use $\si((\Zhatnv{N}{v})_{|v|=n},(\pv{v})_{|v|=n})$-measurability of the coefficient and the chain rule of conditional independence.
    \item In the fourth line we use independence 
between $\Zhatnv{N}{v_0}$ and $\si((\Zhatnv{N}{v})_{|v|=n, v\ne v_0},(\pv{v})_{|v|=n})$.

\end{itemize}
Therefore, because $\Zhatnv{N}{v_0}$ under $\Pa$ has the same distribution as $\ZhatnN nN$ under $\Patr$, and $\Zhatv{a,v_0}$ is independent from the rest of the random variables,
\begin{align*}
    \Qatr  (A)&= 
    \E\ba{f_n(\wh Z^a)\cdot  \Patr (A\mid \Zhatn{N-n})}&\implies\dd{\Qatr }{\Patr}
    &= f_n(\wh Z^a).\\
    \text{where }
    f_n(z) &= 
    \E\ba{\fc{2^n \pv{v_0} z}{\pv{v_0}z + \sum_{w\in \{0,1\}^n\bs \{v_0\}} \pv{w} \Zhatv{w}}}
\end{align*}
Because the function inside the expectation is monotonically increasing in $z$, $f_n$ is monotonically increasing. 


\end{proof}

\subsection{Contiguity for the infinite CREM}
\label{s:contiguity}

The following theorem says that for the infinite CREM, the sequence of (averaged) distributions encountered after sampling some coordinates is contiguous with the CREM.

\begin{thm}\label{t:contiguous}
Consider the CREM with unnormalized covariance function $a:\R_{\ge 0}\to \R_{\ge 0}$, and suppose $\amx = \sup a'$, $\be < \sfc{2\ln 2}{\amx}$. Then 
    \begin{enumerate}
        \item The sequence of distributions of $\ZhatnN{n}{\iy}$ under $\Qatri$ is tight. 
        \item 
        $(\Qatri)_{n\ge 0}$ is contiguous with $(\Patri)_{n\ge 0}$.
    \end{enumerate}
\end{thm}
\begin{proof}
    For a value of $M$ to be chosen, 
    let $\ptv{v} = \pv{v}$ when $\max_{|v|=n}\pv{v} \le M$ and $\ptv{v} = \rc{2^n}$ otherwise. 
We have by definition of $\Qatri$ and a coupling argument that
\begin{align*}
    &\Qatri (\ZhatnN{n}{\iy}\ge L) \\
    &= \Eatri\ba{\fc{\sum_{|v|=n} \pv{v} h_{\ge L} (\Zhatv{v})}{\sum_{|v|=n} \pv{v} \Zhatv{v}}}\\
    &= \Pai\pa{\max_{|v|=n} \pv{v} \ge M}
    + \Eatri\ba{\fc{\sum_{|v|=n} \ptv{v} h_{\ge L} (\Zhatv{v})}{\sum_{|v|=n}\ptv{v} \Zhatv{v}}}\\
    &\le  \Pai\pa{\max_{|v|=n} \pv{v} \ge M}
    + \ep \Pai\ba{\fc{\sum_{|v|=n} \ptv{v} h_{\ge L} (\Zhatv{v})}{\sum_{|v|=n}\ptv{v} \Zhatv{v}} \le \ep}
    + \Pai\ba{\fc{\sum_{|v|=n} \ptv{v} h_{\ge L} (\Zhatv{v})}{\sum_{|v|=n}\ptv{v} \Zhatv{v}} > \ep}.
\end{align*}
The second term goes to 0 as $\ep\to 0$. 
For fixed $\ep$, the third term goes to 0 uniformly in $n$, as $M\to 0$ and $L\to \iy$ by \Cref{l:big-contrib}. For fixed $M$, the first term goes to 0 as $n\to \iy$ by \Cref{c:max-pv}. Thus, $\Qatri (\ZhatnN{n}{\iy}\ge L)$ is bounded (uniformly in $n$) by a function of $L$ going to 0 as $L\to \iy$. This shows tightness.

Now let $A_n$ be a sequence such that $\Patri(A_n)\to 0$. Then for any $L$, letting $f_n$ be as in \Cref{l:com},
\begin{align*}
    \Qatri(A_n) 
    &\le \Qatri\pa{\{\ZhatnN{n}{\iy} \le L\} \cap A_n} + 
     \Qatri\pa{\{\ZhatnN{n}{\iy}  > L\}}\\
    &\le  \Patri(A_n) \cdot f_n(L) + \Qatri\pa{\{\ZhatnN{n}{\iy}  > L\}}.
\end{align*}
    By tightness, as $L\to \iy$, $ \Qatri\pa{\{\ZhatnN{n}{\iy}  > L\}}\to 0$ uniformly in $n$. For any $L$, $f_n(L)$ is bounded uniformly in $n$. Hence $\Qatri(A_n)\to 0$ as $n\to \iy$, showing contiguity.
\end{proof}

\subsection{Quantitative change-of-measure for finite CREM}
\label{s:com-finite}

We now make this quantitative for finite-depth trees.


\begin{lem}\label{l:com-finite}
    Under \Cref{a:hi-temp}(\ref{i:1}) ($\be<\bmin$), 
    the following hold.
\begin{enumerate}
    \item 
    We have the bound
\begin{align*}
\Qatr(\ZhatnN{n}{N}\ge z) & \le 
\exp\pa{-\Om\pf{{g}^4 n}{\amx}}
        + \exp\pa{-{g}^{O(1/g)} e^{\Om({g}^2n)}}
        + {g}^{-O(1/g)} z^{\Om(1/g)}\\
\shortintertext{when $n\ge \fc{C\amx}{{g}^4}\ln \pf{\amx}{g}$ and}
\Qatr(\ZhatnN{n}{N}\ge z) & \le 
2^n\Patr(\Zhatn{N}\ge z)
\end{align*}
for all $n$.
    \item For $0<\ep<1$, $\de = \pf{\ep g}{N\amx}^{\Om\pf{\amx}{{g}^4}}$, the following holds:
    \[
            \text{If }
            \Patr(A)\le \de,
            \text{ then }
            \Qatr(A) \le \ep.
    \]
\end{enumerate}
\end{lem}
\begin{proof}
Let $f_n$ be as in \Cref{l:com}.
        Let $M=e^{-c_1{g}^2n}$ where $c_1$ is as in \Cref{l:ratio-1}(2). Let $\ptv{v} = \pv{v}$ when $\max_{|v|=n}\pv{v} \le M$ and $\ptv{v} = \rc{2^n}$ otherwise. 
By \Cref{l:Z-big} (in conjunction with \Cref{l:renorm}), we have $\Pj\pa{\ZhatnN{n}{N}\ge \ep_1}\ge \ep_2$ for $\ep_1=\rc 2$ and some $\ep_2= {g}^{O(1/g)}$. 
We have by a coupling argument that
    \begin{align*}
    \Qatr(\ZhatnN{n}{N}\ge z) 
    &= \E_{\be, N-n}^{a}\ba{\fc{\sum_{|v|=n} \pv{v} h_{\ge z} (\wh Z^v_{\beta, N-n})}{\sum_{|v|=n} \pv{v} \wh Z^v_{\beta, N-n}}}\\
    &= \Pa\pa{\max_{|v|=n} \pv{v} \ge M}
    + \E_{\be, N-n}^{a}\ba{\fc{\sum_{|v|=n} \ptv{v} h_{\ge z} (\wh Z^v_{\beta, N-n})}{\sum_{|v|=n}\ptv{v} \wh Z^v_{\beta, N-n}}}\\
    &\le 
    e^{-\fc{c_2{g}^4n}{\amx}}
    + 
    \exp\pa{-\fc{\ep_2^2}{2e^{-c_1{g}^2n}}}
    + \fc{4C}{\ep_1\ep_2 z^{p-1}}
\end{align*}from plugging in the bounds in \Cref{l:ratio-1}(2) (in conjunction with \Cref{l:renorm}) and
\Cref{l:big-contrib}, when 
$n\ge \fc{C\amx}{{g}^4}\ln \pf{\amx}{g}$. Substituting in $\ep_1$, $\ep_2$ gives the first part of item 1. Noting that $f_n(z)\le 2^n$ gives the second part of item 1. 

For item 2, note that by 
    Markov's inequality, 
    \[
\Patr\pa{\ZhatnN{n}{N}\ge \rc \de}
\le \de.
    \]
    Let $z=\rc \de$ and  $L = f_n(z)$; because $f_n$ is increasing by \Cref{l:com}, 
    \[
\ZhatnN{n}{N}\ge z \iff \dd{\Qatr}{\Patr}\ge L. 
    \]
    Then for any $A$ such that $\Patr(A)\le \de$, we have the stochastic domination relation under $\Patr$:
    \[
\dd{\Qatr}{\Patr}\one_A \preceq 
\dd{\Qatr}{\Patr}\one_{\dd{\Qatr}{\Patr}\ge L}
= \dd{\Qatr}{\Patr}\one_{\ZhatnN{n}{N}\ge z}.
   \]
   Therefore,
   \begin{align*}
       \Qatr(A) &= \int_{A}
       \dd{\Qatr}{\Patr}\,d\Patr\\
       &\le 
        \int_{\wh Z_{\be, N-n}\ge z}
        \dd{\Qatr}{\Patr}\,d\Patr
        = \Qatr (\wh Z_{\be, N-n}\ge z)\\
        &\le \exp\pa{-\Om\pf{{g}^4 n}{\amx}}
        + \exp\pa{-{g}^{O(1/g)} e^{\Om({g}^2n)}}
        + {g}^{-O(1/g)} z^{-\Om(1/g)}\\
        &\le \exp\pa{-\Om\pf{{g}^4 n}{\amx}}
        + \exp\pa{-{g}^{O(1/g)} e^{\Om({g}^2n)}}
        + {g}^{-O(1/g)} \de^{\Om(1/g)}.
   \end{align*}
   when $n = \Om \pa{\fc{\amx}{{g}^4}\ln \pf{\amx}{g}}$.
   To make this $\le \fc{\ep }{2N}$, it suffices
   for $n=\Om\pa{\fc{\amx}{{g}^4}\ln \pf N\ep}$, $n=\Om\pa{\rc{{g}^2} \pa{\rc{g} \ln \prc{g} + \ln \ln \fc N\ep}}$, and $\de={g}^{O(1)} \pf{\ep}{N}^{O(g)}$. Putting the conditions on $n$ together, it suffices for $n=\Om\pa{\fc{\amx}{{g}^4}\ln \pf{N\amx}{\ep g}}$ for this bound. 
   For $n= O\pa{\fc{\amx}{{g}^4}\ln \pf{N\amx}{\ep g}}$, we have the bound
   \[
\Qatr(A) \le 2^n \de,
   \]
   so it suffices for $\de = \pf{\ep g}{N\amx}^{\Om\pf{\amx}{{g}^4}}$. 
\end{proof}

Generically, the ability to estimate the partition function under the CREM combined with a change-of-measure bound implies that we can also estimate the partition functions for the tilted measure after fixing some coordinates, and hence allows us to sequentially sample from the CREM Gibbs measure. The proof is a sequential coupling argument.
\begin{lem}\label{l:seq-helper}
Consider a CREM with unnormalized covariance function $a(x)$.
    Suppose that 
    the following hold.
    \begin{enumerate}
        \item (Good approximation under CREM) Under the CREM with unnormalized covariance function $a$, we have an algorithm that computes an approximation $\wt Z_{\be, N-n}$ where 
        \[
\ab{\fc{\ZtnN{n}{N}}{\ZhatnN{n}{N}}-1}\le \fc{\ep}{4N}
        \]
        with probability at least $1-\fc{\de}2$.
        \item (Change of measure bound)
        The following holds for all events $A$:
        \begin{align*}
            \text{If }
            \Patr(A)\le \de,
            \text{ then }
            \Qatr(A) \le \fc{\ep}{2N}.
        \end{align*}
    \end{enumerate}
    Then if we run \Cref{a:seq} using the approximation algorithm, then 
\[
\Ean{N} \TV(\wh \mu, \mun{N}) \le \ep. 
\]
\end{lem}
\begin{proof}
    Let $\wh \mu_n$ be the distribution of the vertex $\wh v$ with $|\wh v|=n$, obtained at the $n$th step of the algorithm.
    We use a coupling argument to inductively show that 
    \[
\E \TV(\wh \mu_n, \muNn{n}) \le \fc{\ep n}{N}. 
    \]
    This holds for $n=0$.
    For the induction step, consider a coupling between $\wh v_n\sim \wh \mu^n$ and $v_n\sim \mu_{\be}^n$ such that $\Pj(\wh v_n \ne v_n)\le \fc{\ep n}{N}$. Consider the bad events 
    \[B_x = 
    \bc{
    \ab{\fc{\wt Z^{vx}_{\be}}{\wh Z^{vx}_{\be}}-1}> \fc{\ep}{4N}
    }, \quad x\in \{0,1\}.\]
    By the first assumption applied to $n+1$, 
    when $v$ is fixed, $B_0\cup B_1$ has probability at most $\de$. 
    By the second assumption, for $v=v_n\sim \muNn{n}$, $B_0\cup B_1$ occurs with probability at most $\fc{\ep}{2N}$. 
    Excluding this bad event, by \Cref{l:approx-frac}, 
    \begin{align}
    \label{e:est-next-prob}
    \ab{
    \fc{p_{v0} \wt Z^{v0}_\be}{p_{v0} \wt Z^{v0}_\be + p_{v1} \wt Z^{v1}_\be} - 
    \fc{p_{v0} \wh Z^{v0}_\be}{p_{v0} \wh Z^{v0}_\be + p_{v1} \wh Z^{v1}_\be}} \le \fc{\ep}{N}\end{align}
    and letting $\wh v_{n+1}$, $v_{n+1}^*$ be obtained as $\wh v_n x$ and $v_nx$, respectively, where $x=0$ with probability $\fc{p_{v0} \wt Z^{v0}_\be}{p_{v0} \wt Z^{v0}_\be + p_{v1} \wt Z^{v1}_\be}$, we can couple so that $\Pj(\wh v_{n+1}\ne  v_{n+1}^*)\le \fc{\ep n}{N}$. By \eqref{e:est-next-prob}, we can also couple $v_{n+1}^*$ and $v_{n+1}\sim \muNn{n+1}$ so that $\Pj(v_{n+1}^*\ne  v_{n+1})\le \fc{\ep}N$. This completes the induction step.
\end{proof}

\begin{proof}[Proof of \Cref{t:main-hi-temp} with sequential sampler]
    We verify the two assumptions in \Cref{l:seq-helper}.
    
    The first assumption follows from taking 
    $\ZtnN{n}{N} = \ZhatnN{n}{n+m}$ 
    for 
    \begin{align}\label{e:m}
    m\ge C\pa{\amx {g}^{-4} \ln \pf 1{g\de} + {g}^{-2} \ln \pf {N}\ep}
    \end{align}
    for an appropriate constant $C$ (where we set $\Zhatn{N'} = \Zhatn{N}$ for $N'\ge N$), and using \Cref{t:main-Z} together with \Cref{l:renorm}.
    
    For the second assumption, by \Cref{l:com-finite}, it suffices to take $\de = \pf{\ep g}{N\amx}^{\Om\pf{\amx}{g^4}}$.
    Plugging back in \eqref{e:m}, it suffices for
   \[
m \ge C \amx^2 g^{-8} \ln \pf{N\amx}{\ep g}
   \]
   for large enough constant $C$. Applying \Cref{l:seq-helper} then gives
   \[
\Ean{N} \TV(\wh \mu, \mun{N})\le \ep.
   \]
   Replacing $\ep$ by $\ep\de$, we obtain that $\Ean{N} \TV(\wh \mu, \mun{N})\le \de\ep$, and by Markov's inequality, $\E \TV(\wh \mu, \mun{N})\le \ep$ with probability at least $1-\de$. 

   Finally, note that the time and query complexity to sample each coordinate is $2^{m} = \pf{N\amx}{\ep g}^{O(\amx^2/g^8)}$, and multiplying by $N$ is a negligible factor.
\end{proof}


\section{Conclusion}

For the CREM at high temperature $\be<\bmin$, we gave two efficient algorithms for sampling from the Gibbs measure, based on a Markov chain and sequential sampling procedure. The dependence of the running time is the desired TV error and failure probability is algebraic. Contrary to many sampling results based on Markov chains, we note that geometric convergence does not hold for our Markov chain because the spectral gap is exponentially small. 
This indicates a possible barrier for showing efficient sampling from more complex spin glass models using popular techniques: efficient sampling (with algebraic dependence) may be possible even if a spectral gap or standard functional inequality does not hold. We hope that a complete analysis of the CREM can shed light on the problem of sampling from more complex spin glass models.

Our algorithm assumes access to the intermediate values in the tree $X_u$, $u\in \T_N$. A natural extension is to suppose that the algorithm only has access to $X_u$ for $|u|=N$. We conjecture that this problem can be tackled by taking averages of $X_{wv}$, $|v|=N-n$ as a proxy for $X_w$, $|w|=n$. 

The main question we leave open is what sampling guarantees are possible when $\be>\bmin$. 
For non-concave $A$, \cite{Ho23a} shows that for $\be>\be_G$, no subexponential algorithm can give a sublinear KL divergence guarantee. This leaves a gap between $\be_1$ and $\be_G$ where an efficient algorithm for sublinear KL divergence, but not $\ep$ TV distance, is known. In the case of concave $A$, $\be_G=\iy$, so the region is the entirety of the region $\be>\bmin$. 
It remains open to determine (1) what the threshold is at which constant TV distance is achievable, and (2) what the Pareto frontier is for (super-polynomial) running time and (super-constant) accuracy in KL divergence.
For (1), we note that in contrast to \cite{Ho23a}, our methods rely critically on fluctuation results which only hold for $\be<\be_c$, so we believe that TV distance guarantees are not possible for $\be>\be_c$, even if $\be_G>\be_c$. 
For (2), based on the result relating the trajectory of extremal particles in branching random walks to Brownian excursions \cite{chen2015scaling,chen2019trajectory}, we make the following more precise conjecture.

\begin{conj}
Consider a CREM with concave covariance function $A$. 
    For $\be>\be_c$, $\al\le \rc 2$ and $\ep>0$, there is a $2^{\wt O_\ep(N^\al)}$-time algorithm to, with high probability 
    \begin{enumerate}
        \item (Optimization) find 
        $\wh v$ such that $X_{\wh v}\ge \OPT - \ep N^{1-2\al}$, where $\OPT = \max_{|v|=N} X_v$. 
        \item (Sampling) 
        output a sample from a distribution $\wh \mu$ such that 
        $\KL(\wh \mu\|\mun{N})\le \ep N^{1-2\al}$. 
    \end{enumerate}
\end{conj}

\section*{Acknowledgements}

This work grew out of discussions from the Random Theory 2023 workshop. We are grateful for Fu-Hsuan Ho's feedback on the first draft.

\printbibliography

\appendix


\section{Calculations}

\begin{lem}[Gaussian tail bound]
\label{l:gtail}
    For $\xi\sim \cal N(0,1)$, we have
\begin{align}
\label{e:gaussian-tail-upper}
\Pj(\xi\ge u) &\le \rc{\sqrt{2\pi}u}e^{-u^2/2}\le e^{-u^2/2}\\
\label{e:gaussian-tail-lower}
\Pj\pa{\xi \ge t}&\ge \rc{\sqrt{2\pi e^3}t} e^{-\rc 2 t^2}\ge \rc{12t} e^{-\rc 2 t^2}.
    \end{align}
\end{lem}
\begin{proof}
For the right side of \eqref{e:gaussian-tail-upper}, note that for $u\ge \rc{\sqrt{2\pi}}$, this holds, and for $0\le u\le \rc{\sqrt{2\pi}}$, $\Pj(\xi\ge u)\le \rc 2 \le e^{-\rc 2 \prc{2\pi}}$.
\end{proof}
\begin{lem}\label{l:eat-blnt}
    Let $a,b,c>0$. If $2b\ge a$ and $t\ge \fc{6b}{a}\ln \fc{2b}a \vee \fc{2}{a}\ln c$, then 
    \[
e^{at - b\ln t} = \fc{e^{at}}{t^b} \ge c.
    \]
\end{lem}
\begin{proof}
    We have
    \begin{align*}
        e^{at - b\ln t}
        &= e^{\fc{at}2} e^{\fc{at}2 - b\ln t} \\
        &\ge ce^{\fc{at}2 - b\ln t} .
    \end{align*}
    Now $\fc{at}2 - b\ln t$ achieves maximum when $\fc a2 = \fc{b}t$, i.e., $t=\fc{2b}a$, and is increasing for $t\ge \fc ba$. Thus 
    \begin{align*}
        \fc{at}2 - b\ln t\ge 
        3b\ln \fc {2b}a - b\ln \pa{\fc{6b}{a}\ln \fc ba}
        \ge 2b \ln \fc{b}{a} - b \ln \pa{3\ln \fc ba}
    \end{align*}
    so 
    \begin{align*}
        e^{\fc{at}2 - b\ln t}
        \ge \pa{\pf{b}{a}^2 / \pa{3\ln \fc ba}}^b \ge 1
    \end{align*}
    using the fact that $x^2 \ge 3\ln x$ for $x>0$ (by noting that it attains maximum at $x=\sfc 32$ and $\fc{3}2 \ge \fc 32 \ln \pf 32$). 
\end{proof}

\begin{lem}\label{l:approx-frac}
    If $A,B>0$ and $0<\ep\le \rc 2$ are such that 
    $\ab{\fc{\wh A}{A}-1}\le \ep$ and $\ab{\fc{\wh B}{B}-1}\le \ep$, 
    then for any $p,q>0$,
    \begin{align*}
        \ab{\fc{p\wh A}{p\wh A + q\wh B} - \fc{pA}{pA+qB}}\le 2\ep. 
    \end{align*}
\end{lem}
\begin{proof}
By replacing $p\wh A$ with $\wh A$ and $q\wh B$ with $\wh B$, we may assume without loss of generality that $p=q=1$. 
    We have
    \begin{align*}
        \ab{\fc{\wh A}{\wh A + \wh B} - \fc{A}{A+B}}
        &= \fc{(\wh A - A)B + (B-\wh B) A}{(\wh A + \wh B)(A + B)}= \fc{\fc{\wh A-A}{A} + \fc{B-\wh B}{B}}{\fc{(\wh A + \wh B)(A+B)}{AB}}.
    \end{align*}
    The numerator is at most $2\ep$ in absolute value, and the denominator satisfies
    \begin{align*}
        \fc{(\wh A + \wh B)(A+B)}{AB}
        &\ge (1-\ep) \fc{(A+B)^2}{AB} 
        \ge  \fc{(A+B)^2}{2AB} \ge 1.
    \end{align*}
    The bound follows.
\end{proof}


\renewcommand\nomgroup[1]{%
  \item[\bfseries
  \ifstrequal{#1}{A}{Probability measures}{%
  \ifstrequal{#1}{B}{Random measures}{%
  \ifstrequal{#1}{C}{Partition functions}{}}}%
]}
\nomenclature[A, 01]{\(\Pa\)}{CREM disorder measure, \Cref{d:crem}}
\nomenclature[A, 02]{\(\Patr \)}{$= \Pj_{N-n}^{\bfT_n a}$, CREM disorder measure with translated $a$, \Cref{d:dists}}
\nomenclature[A, 03]{\(\Qatr\)}{Tilted CREM disorder measure, \Cref{d:dists}}
\nomenclature[A, 04]{\(\Pai\)}{Infinite CREM disorder measure, \Cref{d:crem-iy}}

\nomenclature[B, 01]{\(\mun{n}\)}{Gibbs measure on CREM, \Cref{d:crem}}
\nomenclature[B, 02]{\(\munv{n}{v}\)}{Gibbs measure on CREM rooted at $v$, \Cref{d:dists}}
\nomenclature[B, 03]{\(\muNn{n}\)}{Marginal distribution of first $n$ coordinates of $\mun{N}$, \Cref{d:dists}}
\nomenclature[B, 04]{\(\nun{n}\)}{$=\munv{N}{v}$ averaged over $v$, where $v\sim \muNn{n}$, \Cref{d:dists}}

\nomenclature[C, 01]{\(\Zn{n}\)}{Partition function at level $n$, \Cref{d:crem}}
\nomenclature[C, 02]{\(\Znv{n}{v}\)}{Partition function at level $n$, rooted at $v$, \eqref{e:Znv}}
\nomenclature[C, 03]{\(\Zhatn{n} = \Zhatan{n}\)}{Normalized partition function at level $n$, \Cref{d:crem}}
\nomenclature[C, 04]{\(\Zhatnv{n}{v} = \Zhatnv{n}{a,v}\)}{Normalized partition function at level $n$, rooted at $v$, \eqref{e:Zhatnv}}
\nomenclature[C, 05]{\(\ZhatnN{n}{N} = \ZhatanN{n}{N}\)}{$=\Zhatnv{N-n}{\bfT_n a}$ Normalized partition function with translated $a$ at level $N-n$, \Cref{d:crem-iy}}
\nomenclature[C, 07]{\(\Ztn{n}\)}{Estimate for $\Zhatn{N}$ in \Cref{l:seq-helper}}
\nomenclature[C, 08]{\(\Ztnv{n}{v}\)}{Estimate for $\Zhatnv{N}{v}$ in \Cref{a:seq}}

\section{List of notations}
\vspace{-1cm}
\printnomenclature

\label{s:nomen}

\end{document}